\documentclass{article}
\usepackage[margin=1in]{geometry}
\usepackage{graphicx,amsmath,amsthm,comment,framed,algorithm,algorithmic,subcaption,url,bm}
\usepackage{lipsum}
\usepackage{amsfonts}
\usepackage{graphicx}
\usepackage{epstopdf}
\usepackage{algorithmic}
\usepackage{amsmath,comment,framed,algorithm,url,bm,bbm}
\title{On the simplicity and conditioning of \\ low rank semidefinite programs}
\author{Lijun Ding\footnote{School of Operations Research and Information Engineering, Cornell University,
		Ithaca, NY 14850, USA; \texttt{ld446@cornell.edu}} \;and Madeleine Udell\footnote{School of Operations Research and Information Engineering, Cornell University,
		Ithaca, NY 14850, USA;
		\texttt{udell@cornell.edu}}}
\usepackage[usenames,dvipsnames]{xcolor}
\definecolor{purple}{rgb}{0.3,0.0,.4}

\usepackage{bbm}
\usepackage{thmtools}
\usepackage{thm-restate}
\usepackage{hyperref}
\usepackage[capitalise]{cleveref}
\usepackage{xspace}
\usepackage{amssymb}% http://ctan.org/pkg/amssymb
\usepackage{pifont}% http://ctan.org/pkg/pifont
\usepackage{accents}

\declaretheorem[name=Theorem,numberwithin=section]{thm}
\newcommand{\newcontent}[1]{#1}
\newcommand{\defn}{:\,=}
\newcommand{\BEAS}{\begin{eqnarray*}}
\newcommand{\EEAS}{\end{eqnarray*}}
\newcommand{\BEA}{\begin{eqnarray}}
\newcommand{\EEA}{\end{eqnarray}}
\newcommand{\BEQ}{\begin{equation}}
\newcommand{\EEQ}{\end{equation}}
\newcommand{\BIT}{\begin{itemize}}
\newcommand{\EIT}{\end{itemize}}
\newcommand{\BNUM}{\begin{enumerate}}
\newcommand{\ENUM}{\end{enumerate}}

\newcommand{\beas}{\begin{eqnarray*}}
\newcommand{\eeas}{\end{eqnarray*}}
\newcommand{\bea}{\begin{eqnarray}}
\newcommand{\eea}{\end{eqnarray}}
\newcommand{\beq}{\begin{equation}}
\newcommand{\eeq}{\end{equation}}
\newcommand{\bit}{\begin{itemize}}
\newcommand{\eit}{\end{itemize}}
\newcommand{\ben}{\begin{enumerate}}
\newcommand{\een}{\end{enumerate}}

\newcommand{\ba}{\begin{array}}
\newcommand{\ea}{\end{array}}
\newcommand{\bbm}{\begin{bmatrix}}
\newcommand{\ebm}{\end{bmatrix}}

\newcommand{\eg}{e.g., }
\newcommand{\ie}{i.e., }

\newcommand{\ones}{\mathbf 1}

\newcommand{\reals}{{\mbox{\bf R}}}
\newcommand{\integers}{{\mbox{\bf Z}}}

\newcommand{\sym}{{\mbox{\bf S}}}  % symmetric matrices

\newcommand{\Id}{\mathcal{I}}

\newcommand{\range}{\mathop{\bf range}}
\newcommand{\rank}{\mathop{\bf rank}}
\newcommand{\nullspace}{{\mathop {\bf null}}}
\newcommand{\tr}{\mathop{\bf tr}}

\newcommand{\diag}{\mathop{\bf diag}}
\newcommand{\ddiag}{\mathop{\bf {ddiag}}}

\newcommand{\inprod}[2]{\langle #1,#2\rangle}

\newcommand{\twonorm}[1]{\left\|#1\right\|_2}
\newcommand{\norm}[1]{\left\|#1\right\|}
\newcommand{\fronorm}[1]{\left\|#1\right\|_{\mbox{\tiny{F}}}}
\newcommand{\opnorm}[1]{\left\|#1\right\|_{\mbox{\tiny{\textup{op}}}}}
\newcommand{\nucnorm}[1]{\left\|#1\right\|_*}
\newcommand{\infnorm}[1]{\left\|#1\right\|_\infty}

\newcommand{\Prob}{\mathop{\bf Prob}}

\newcommand{\ri}{\mathop{\bf ri}}

\newcommand{\indicator}{\chi}
\newcommand{\grad}{\mbox{grad}}
\newcommand{\hess}{\mbox{Hess}}

\theoremstyle{definition} % don't italicize theorems
\newtheorem{theorem}{Theorem}
\newtheorem{lemma}{Lemma}
\newtheorem{corollary}{Corollary}

\newtheorem{definition}{Definition}

\newcommand{\bigO}[1]{\mathcal{O}(#1)}

\newcommand{\pval}{p_\star}
\newcommand{\dval}{d_\star}
\newcommand{\dm}{n}
\newcommand{\cons}{m}
\newcommand{\dg}{d}
\newcommand{\block}{\mbox{Block}}
\newcommand{\rsol}{r_\star}
\newcommand{\xsol}{X_\star}
\newcommand{\ysol}{y_\star}

\newcommand{\zsol}{Z(\ysol)}%{Z_\star}
\newcommand{\zsoloy}{Z_\star }

\newcommand{\pertA}{\Delta \Amap}
\newcommand{\pertb}{\Delta b}
\newcommand{\pertC}{\Delta C}

\newcommand{\Amap}{\mathcal{A}}

\newcommand{\Admap}{\mathcal{A}^*}

\newcommand{\trux}{X_\natural}
\newcommand{\projt}{\Pi_{\mathcal{T}}}
\newcommand{\projto}{\Pi_{\mathcal{T}^\perp}}
\newcommand{\projomega}[1]{\Pi_{\Omega_{#1}}}

\newcommand{\hpk}[1]{\mathcal{H}_{\Omega_{#1}}}
\newcommand{\indic}{\mathbbm{1}}
\newcommand{\rprojk}[1]{\mathcal{R}_{\Omega_{#1}}}
\newcommand{\To}{t_0}
\newcommand{\inco}{\mu}
\newcommand{\ob}{\Omega}
\newcommand{\trur}{r_{\natural}}
\newcommand{\tY}{\tilde{Y}}
\newcommand{\tXsol}{\tilde{\xsol}}
\newcommand{\tX}{\tilde{X}}
\newcommand{\truA}{\tilde{A}}

\begin{document}
	\maketitle
\begin{abstract}
	Low rank matrix recovery problems appear widely in
	statistics, combinatorics, and imaging.
	One celebrated method for solving these problems is
	to formulate and solve a semidefinite program (SDP).
	It is often known that the exact solution to the SDP
	with perfect data recovers the solution to the original
	low rank matrix recovery problem.
	It is more challenging to show that an approximate solution to the SDP
	formulated with noisy problem data acceptably solves the original problem;
	arguments are usually ad hoc for each problem setting, and can be complex.
	
	In this note,
	we identify a set of conditions that we call \emph{\newcontent{simplicity}} that
	limit the error due to noisy problem data or incomplete convergence.
	In this sense, \newcontent{simple} SDPs are robust: \newcontent{simple} SDPs can be (approximately)
	solved efficiently at scale;
	and the resulting approximate solutions, even with noisy data, can be trusted.
	Moreover, we show that \newcontent{simplicity}  holds generically,
	and also for many structured low rank matrix recovery problems,
	including the stochastic block model, $\mathbb{Z}_2$ synchronization, and matrix completion.
	Formally, we call an SDP \newcontent{simple} if it has a surjective constraint map,
	admits a unique primal and dual solution pair, and satisfies strong duality and strict complementarity.
	
	However, \newcontent{simplicity}  is not a panacea:
	we show the Burer-Monteiro formulation of the SDP
	may have spurious second-order critical points,
	even for a \newcontent{simple} SDP with a rank 1 solution.
\end{abstract}
\section{Introduction}\label{sec: intro}
We consider a semidefinite program (SDP) in the standard form
\beq\label{p}\tag{\text{$\mathcal{P}$}}
\ba{ll}
\mbox{minimize} & \inprod{C}{X}\\
\mbox{subject to} & \mathcal{A}X =  b
\quad\text{and}\quad
X \succeq 0, \\
\ea
\eeq
where $\langle\cdot,\cdot\rangle$ denotes the matrix trace inner product.
The primal variable is the symmetric positive semidefinite (PSD) matrix
$X \in \sym^\dm_+\subset \reals^{\dm\times \dm}$.
The problem data comprises
a symmetric (but possibly indefinite) cost matrix $C\in \sym^\dm$,
a righthand side $b\in\reals^m$, and
a linear constraint map $\mathcal{A}: \reals^{\dm\times \dm} \rightarrow \reals^m$
with rank $m$ operating on any $H\in \reals^{\dm \times \dm}$ by
$
[\mathcal{A}H]_i = \inprod{A_i}{H},i=1,\dots,m$ for some fixed symmetric $A_i\in \sym^\dm$.  Denote an
arbitrary solution of \eqref{p} as $\xsol$ and the optimal value as $\pval$.

The optimization problem \eqref{p} appears in problems in
statistics \cite{srebro2005rank}, combinatorics \cite{goemans1995improved}, and
imaging \cite{chai2010array}, among others.
Due to the nature of these applications,
practical instances of \eqref{p} such as
matrix completion \cite{srebro2005rank,udell2019why} and MaxCut \cite{goemans1995improved}
are often expected to have low rank solutions.
It is also notable that any instance of \eqref{p} admits a solution
with rank $\rsol$ satisfying $\frac{\rsol(\rsol+1)}{2}\leq \cons$
\cite{barvinok1995problems, pataki1998rank}.

\paragraph{\newcontent{Simplicity} }
Formally, we say an SDP is \emph{simple} if
it has a surjective constraint map,
admits a unique primal and dual %non-degenerate
solution pair,
and satisfies both strong duality and strict complementarity.
(See Section \ref{sec: regularity} for more detail.)
These conditions suffice to guarantee many useful properties about the resulting SDP.

\newcontent{Simplicity}  was found by \cite{alizadeh1997complementarity} to hold \emph{generically}:
for almost all $\Amap$, $b$ and $C$,
\eqref{p} is \newcontent{simple} so long as a primal and dual solution pair exists.
A followup work \cite[Section 5]{drusvyatskiy2016generic} strengthens this result:
for every surjective $\Amap$, \newcontent{simplicity} holds for almost all $b$ and $C$,
again conditioning on the existence of a primal and dual solution pair.

However, realistic applications of semidefinite programming may place
structural constraints on $\Amap$, $b$, and $C$:
for example, in matrix completion, the cost matrix $C = I$;
in MaxCut type SDPs, the constraint map $\Amap=\diag$
and the right hand side $b$ is the vector of all ones.
We will show in Section \ref{sec: Regular low rank SDPs}, and \ref{sec: weakbm} that many of these SDPs,
including $\integers_2$ synchronization and the stochastic block model, are still simple.
We also show in Section \ref{sec: mc} that matrix completion
is \emph{primal simplicity}: it satisfies all conditions for \newcontent{simplicity}  except
for dual uniqueness.

\paragraph{Conditioning and \newcontent{simplicity} }
Many authors have shown that instances of the primal SDP \eqref{p} appearing
statistical or signal processing problems \cite{candes2009exact, waldspurger2015phase,bandeira2018random},
admit a unique low rank solution
which coincides with (or is close to) the underlying true signal.
However, this analysis does not fully solve the original problem:
optimization procedures give reliable solutions only when the problem is \emph{well-conditioned};
otherwise, inaccuracies in the problem data or incomplete convergence can lead to wildly different
reconstructions of the underlying signal.
Here we consider two different notions of problem conditioning:
\begin{enumerate}
	\item \emph{Measurement error:}
	Suppose we obtain perturbed problem data $\Amap+\pertA$, $b+\pertb$, and $C+\pertC$ instead
	of the orignal problem data $\Amap$, $b$, and $C$ due to noisy measurements.
	We solve \eqref{p} with perturbed problem data and obtain a perturbed solution
	$\xsol'$. To ensure that the perturbed solution $\xsol'$ is meaningful for the original problem,
	we must ensure the error in the solution $ \xsol -\xsol'$ is controlled by the
	size of the perturbation $(\pertA,\pertb,\pertC)$ in the data.
	
	We can describe the \emph{sensitivity of the solution} to measurement error
	by finding constants $\alpha,\beta>0$ such that for all small $(\pertA,\pertb, \pertC)$,
	\[
	\fronorm{\xsol -\xsol'}^{\alpha} \leq \beta \left(\norm{\pertA}+\norm{\pertb}+\norm{\pertC}\right).
	\]
	%The pair $(\alpha,\beta)$ describe the sensitivity of the solution to measurement error.
	
	\item \emph{Optimization error}:
	Most optimization algorithms offer guarantees on the suboptimality $\tr(CX)- \pval$
	of the putative solution $X$ they return, % to the primal SDP \eqref{p},
	but many cannot guarantee bounds on the distance to the solution, $\fronorm{X-\xsol}$.
	However, the distance to the solution is usually the more important metric for
	statistical and signal processing applications.
	Hence it is important to understand how (and when) guarantees in suboptimality
	translate into guarantees on the distance to the solution.
	
	We may seek to bound the distance to the solution, $\norm{X-\xsol}$,
	in terms of simpler metrics of optimization error:
	the infeasibility with respect to conic constraints, $(-\lambda_{\min}(X))_+$,
	and linear constraints, $\twonorm{\Amap X-b}$,
	and the suboptimality, $\tr(CX)-\tr(C\xsol)$.
	(Throughout the paper we define $(x)_+ = \max\{0,x\}$ for $x \in \reals$.)
	We produce an \emph{error bound} on the solution
	by finding constants $\gamma,\rho>0$ such that for all $X$ near $\xsol$,
	\[
	\twonorm{X-\xsol}^\rho \leq \gamma \left( \twonorm{\Amap X-b} +(-\lambda_{\min}(X))_++(\tr(CX)-\pval)\right).
	\]
	
\end{enumerate}
The exponents $\rho$ and $\alpha$ and the multiplicative factors $\gamma$ and $\beta$,
can be interpreted as condition numbers of \eqref{p}.

\newcontent{Simple} SDPs obey useful bounds on these condition numbers.
In the literature, it has been found that if the SDP \eqref{p} is simple, then
$\rho=1$ \cite{nayakkankuppam1999conditioning} and $\alpha =2$ \cite{sturm2000error}.
We note that
$\alpha =2$ only requires primal \newcontent{simplicity} .
An upcoming work of ours \cite{DingAndMconicgrowth} shows that
$\rho=2$ under the weaker condition of primal \newcontent{simplicity} .
Estimates of $\beta$ and $\gamma$ for \newcontent{simple} SDPs based on
problem data and solutions are also available respectively in
\cite{nayakkankuppam1999conditioning} and our upcoming work \cite{DingAndMconicgrowth}.
When the SDP \eqref{p} is not \newcontent{simple} but only feasible, then the exponent of
$\rho$ can become as large as $2^{n-1}$ which is shown to be tight \cite[Example 2]{sturm2000error}. In such
cases, the SDP is very ill-conditioned.
Thus if the SDP \eqref{p} is \newcontent{simple} or primal \newcontent{simple},
neither measurement error nor optimization error impede signal recovery,
as the distance to the solution (which is the true signal or close to it) grows at most quadratically in the measurement or optimization error.

\paragraph{\newcontent{Simplicity}  and algorithmic convergence} \newcontent{Simplicity}  also plays an
important role in the convergence analysis of algorithms of SDP. For example:
\begin{itemize}
	\item \newcontent{Simple} SDP can be solved efficiently at scale:
	for example, the storage-optimal algorithm of \cite{ding2019optimal} requires \newcontent{simplicity}  to
	ensure the limit of the dual iterates produces a meaningful approximation of the primal solution $\xsol$.
	\item For \newcontent{simple} SDP, the central path of an interior point method (IPM)
	leads to the analytical center of the solution set
	\cite{halicka2002convergence, luo1998superlinear}.
\end{itemize}
\newcontent{Simplicity}  can also improve the convergence rate for many algorithms:
\begin{itemize}
	\item For SDP that satisfy Slater's condition,
	IPMs can only be shown to converge linearly \cite{nesterov2018lectures}.
	But for primal \newcontent{simple} SDP, IPMs achieve superlinear convergence \cite{luo1998superlinear};
	and for \newcontent{simple} SDP, IPMs achieve quadratic convergence \cite{alizadeh1998primal}.
	\item For the exact penalty formulation of the dual SDP \cite{ding2019optimal},
	subgradient-type methods with constant or diminishing stepsize
	require $\bigO{\frac{1}{\epsilon^2}}$ iterations
	to reach an $\epsilon$-suboptimal dual solution.
	But for \newcontent{simple} SDP, subgradient methods achieve
	faster sublinear rates $\bigO{\frac{1}{\epsilon}}$,
	using the quadratic error bound induced by \newcontent{simplicity}  for the analysis \cite{sturm2000error, johnstone2017faster}.
\end{itemize}

\paragraph{\newcontent{Simplicity}  and the Burer-Monteiro method }
The Burer-Monteiro (BM) \cite{burer2003nonlinear} approach solves the SDP \eqref{p}
by factoring the decision variable, building on earlier work by Homer and
Peinado \cite{homer1997design} that introduced the approach for the MaxCut SDP.
The BM approach factors the decision variable $X = FF^\top$,
with factor $F\in \reals^{\dm \times r}$, and
solves the following (nonconvex) problem:
\begin{equation*} \tag{BM}\label{BM}
	\ba{ll}
	\mbox{minimize} & \tr( CFF^\top ) \\
	\mbox{subject to} & \mathcal{A}(FF^\top) = b.
	\ea
\end{equation*}
When $r$ exceeds the rank of any solution to \eqref{p},
\eqref{BM} and \eqref{p} have the same solution set.

Usually, \eqref{BM} is solved using a Riemannian gradient or trust region method \cite{boumal2018global},
which requires that the feasible set forms a smooth Riemannian manifold.
Following \cite{boumal2018deterministic}, we call such an SDP \emph{smooth}:
the feasible set $\mathcal{A}(FF^\top) = b$ forms a smooth Riemannian manifold.
In this paper we will consider many interesting smooth SDPs:
including MaxCut, OrthogonalCut, and an SDP relaxation of a problem optimizing over a product of spheres;
and statistical problems like $\integers_2$ synchronization and the stochastic block model.
Notice that many interesting large scale SDPs,
such as matrix completion \cite{candes2010power} and phase retrieval \cite{candes2013phaselift},
may not be smooth.

Since these Riemannian optimization methods are only guaranteed to find second order stationary points,
we will say the BM method \emph{succeeds} for a smooth SDP
when all second order stationary points are globally optimal
(and fails otherwise).
A recent result \cite{boumal2018deterministic} shows that
for smooth SDP (and under a few more technical conditions),
for almost all objectives $C$,
BM succeeds if $\frac{r(r+1)}{2}>m$.

Does the BM method succeed for every (smooth) \newcontent{simple} SDP?
Alas, no: we show the Burer-Monteiro approach \eqref{bm} can fail
when $\frac{r(r+1)}{2}+r<m$, even if \eqref{p} is \newcontent{simple}.
This result extends a recent counterexamples due to \cite{waldspurger2018rank}
by showing uniqueness of the dual solution.
Hence storage optimal algorithms for SDP, such as \cite{ding2019optimal},
that operate directly on the SDP (without factorization) have advantages over BM.
\paragraph{Paper organization} We formally define \newcontent{simple} SDPs in Section \ref{sec: regularity}.
Section \ref{sec: notation} introduces the notation used in this paper.
In Section \ref{sec: Regular low rank SDPs}, we show that
every PSD matrix solves a \newcontent{simple} SDP
and that primal \newcontent{simplicity}  holds for almost all objectives $C$ under Slater's condition.
In Section \ref{sec: bmfail}, we construct \newcontent{simple} SDPs for which the Burer-Monteiro approach fails.
In Section \ref{sec: weakbm},
we use \newcontent{simplicity}  to show that the SDPs corresponding to
the stochastic block model and $\mathbb{Z}_2$ synchronization can recover the
ground truth from noisy data.
Notably, we show recovery is possible at higher noise thresholds
than those for which the BM approach is known to succeed.
Finally, in Section \ref{sec: mc},
we show that the celebrated matrix completion SDP is primal simple, but not simple.
\subsection{\newcontent{Simplicity} }\label{sec: regularity}
To start, recall the dual problem of \eqref{p} is
\beq \label{d}\tag{\text{$\mathcal{D}$}}
\ba{ll}
\mbox  {maximize} & \inprod{b}{y} \\
\mbox{subject to}  & C- \mathcal{A}^* y \succeq 0. \\
\ea
\eeq
Here $\inprod{\cdot}{\cdot}$ is the dot product in $\reals^\cons$.
The decision variable is the vector $y\in \reals^\cons$.
The map $\Admap: \reals^m \rightarrow \reals^{\dm \times \dm}$
is the adjoint of the linear map $\Amap$,
which satisfies $\inprod{y}{\Amap X} = \inprod{\Admap y}{X}$.
Explicitly,
$
\Admap(y) = C-\sum_{i=1}^m y_i A_i
$ for $y\in \reals^\cons$.

We now formally state the conditions that define an \newcontent{simple} SDP.
The first two conditions, \emph{strong duality} and \newcontent{\emph{surjective constraint map}},
are standard in the literature.
%\begin{itemize}
%	\item[1.]
\begin{definition}[Strong Duality] \eqref{p} and \eqref{d} satisfy strong duality if
	there is a primal-dual solution pair and for any solution pair $(\xsol,\ysol)\in \sym^{\dm}\times \reals^m$ to \eqref{p}
	and \eqref{d},
	\[
	\pval := \tr(C\xsol) = b^\top \ysol =:\dval.
	\]
\end{definition}
\noindent Notably, strong duality holds under \emph{primal and dual Slater's condition}:
existence of feasible primal $X_0\succ 0$ and dual $y_0$ with $C-\Admap y_0\succ 0$.

\newcontent{Surjectivity of the linear constraint map $\Amap$ ensures that there are no redundant linear constraints.}
\newcontent{
	\begin{definition}[Surjective constraint map] We say \eqref{p} has a surjective constraint map
		if the linear constraint map $\Amap$ is surjective, or equivalently, the matrices $A_i$ are linearly independent in $\sym^\dm$.
	\end{definition}%
}
\noindent \newcontent{Simplicity}  also requires strict complementary slackness. %, or \emph{strict complementarity}.
\begin{definition}[Strict complementarity]
	We say a solution pair $(\xsol,\ysol)\in \sym^{\dm}\times \reals^m$ to \eqref{p} and \eqref{d} is strictly complementary
	if
	\[
	\rank(\xsol) + \rank(C -\Admap \ysol ) = \dm.
	\]
	If the primal \eqref{p} and dual \eqref{d} SDP pair has one strictly complementary solution pair,
	we say the SDP pair satisfies \emph{strict complementarity},
	or simply that the primal SDP \eqref{p} satisfies strict complementarity.
\end{definition}%

Linear programs always have some strictly complementary solution whenever they exist:
there is always some primal optimal $x^\star \in \reals^\dm$ and dual optimal
$z^\star= c-A^\top y^\star \in \reals^\dm$ such that
\[
\text{nnz}(x^\star)+\text{nnz}(z^\star)=n,
\]
where $\text{nnz}$ is the number of nonzeros \cite{goldman1956theory}.
In contrast, semidefinite programs may not satisfy strict complementarity \cite[Example on page 117]{alizadeh1997complementarity}.

Finally, \newcontent{simplicity}  requires that both \eqref{p} and \eqref{d} have unique solutions.
\begin{definition}[\newcontent{Simple SDP}] The SDP \eqref{p} is simple
	if
	\begin{itemize}
		\item[1.] \eqref{p} satisfies strong duality;
		\item[2.] \newcontent{\eqref{p} has a surjective constraint map};
		\item[3.] \eqref{p} satisfies strict complementarity; and
		\item[4.] \eqref{p} and \eqref{d} both have unique solutions.
	\end{itemize}
\end{definition}

Next, we introduce primal \newcontent{simplicity} .
\begin{definition}[Primal simple SDP] The SDP \eqref{p} is \emph{primal simple}
	if it has \newcontent{a surjective constraint map}, and satisfies strong duality and strict complementarity,
	and the solution to \eqref{p} is unique.
\end{definition}

The dual of a primal \newcontent{simple} SDP may admit multiple solutions.
Notice every \newcontent{simple} SDP is primal simple.
Primal \newcontent{simplicity}  is practically important:
for example, the matrix completion SDP \cite{candes2009exact},
introduced in Section \ref{sec: mc},
is primal \newcontent{simple} but not simple.
Primal \newcontent{simple} SDPs inherit some (but not all) of the nice properties
of \newcontent{simple} SDPs.

%\paragraph{Equivalent conditions}
%In the definition of \newcontent{simplicity} , uniqueness of the primal and dual solutions may be replaced by
%nondegeneracy as defined in \cite{alizadeh1997complementarity}.
%Indeed, an SDP is \newcontent{simple} if and only if it has surjective constraint map,
%and satisfies
%strong duality, strict complementarity,
%and nondegeneracy.

%We emphasize primal and dual uniqueness in our definition because as the condition is simpler to state.

\paragraph{\newcontent{Simplicity}  under generic problem data} As mentioned in Section \ref{sec: intro},
for almost all $\Amap,C,b$ (under the Lebesgue measure),
if the SDP pair with problem data $\Amap,C,b$ has a primal and a dual solution
then it is \newcontent{simple} \cite[Theorem 11, 14 and 15]{alizadeh1997complementarity}.
In this paper, we also show in Theorem \ref{thm: Primal regular SDP for generic C}
that for fixed $\Amap$ and $b$, the SDP pair is primal \newcontent{simple} for almost all $C$.

\newcontent{\subsubsection{Relationship with conditions defined in the literature}
	Simple SDPs are related to many other ideas proposed earlier in the SDP literature. We review their connections here to clarify the terminology.
	In the following discussion, we assume that \eqref{p}
	satisfies strong duality and has a surjective constraint map.
	
	First, nondegeneracy as defined in \cite{alizadeh1997complementarity}
	is well-known in the interior point methods community,
	and consists of primal nondegeneracy \cite[Definition 5]{alizadeh1997complementarity}
	and dual nondegeneracy \cite[Definition 8]{alizadeh1997complementarity}.
	These same two conditions are called primal and dual \emph{constraint} nondegeneracy
	in the variational analysis community; see e.g.
	\cite[Definition 2.1]{robinson2003constraint} and
	\cite[Definition 8]{chan2008constraint}.
	Primal and dual nondegenerate SDPs are not necessarily simple,
	as these two conditions do not imply strict complementarity
	\cite[Example on page 117]{alizadeh1997complementarity}.
	However, primal nondegeneracy
	does imply dual uniqueness \cite[Theorem 7]{alizadeh1997complementarity},
	and conversely dual nondegeneracy implies primal uniqueness.
	
	Second, the term \emph{strong regularity}
	appears in the study of generalized equations \cite{robinson1980strongly}.
	It has been shown in \cite[Theorem 18]{chan2008constraint} that
	strong regularity for the KKT equation of \eqref{p} is
	equivalent to primal and dual nondegeneracy.
	As noted before, primal and dual nondegeneracy together imply
	primal and dual uniqueness.
	Further, under the assumption of strict complementarity,
	they are equivalent~\cite[Theorem 11]{alizadeh1997complementarity}.
	A simple SDP satisfies
	strong regularity (of the KKT equation);
	the converse is false due to \cite[Example on page 117]{alizadeh1997complementarity}.
	If \eqref{p} is primal simple but has multiple dual solutions,
	then strong regularity (for the KKT equation) fails:
	under strict complementarity, primal nondegeneracy is equivalent to dual uniqueness  \cite[Theorem 7]{alizadeh1997complementarity}.}

\subsection{Notation} \label{sec: notation}
\paragraph{Norms and Eigenvalues} For a matrix $B\in \reals^{\dm_1\times \dm_2}$,
we denote its Frobenius, operator two norm, and nuclear norm (sum of singular values)
as $\fronorm{B},\opnorm{B}$, and $\nucnorm{B}$ respectively.
The operator norm of a linear operator $\mathcal{B}: \reals^{\dm_1\times \dm_2} \rightarrow \reals^{\dm_1\times \dm_2}$
is defined as $\opnorm{\mathcal{B}}=\max_{A\in \small \reals^{\dm_1\times \dm_2}}\fronorm{\mathcal{B}(A)}$.
We write the eigenvalues of a symmetric matrix $A\in \sym^\dm$ in decreasing order as
\[
\lambda_1(A)\geq \lambda_2(A) \dots\geq\lambda_\dm (A).
\]
We define the singular values $\sigma_i:\reals^{n_1\times n_2}\rightarrow \reals$ similarly.
%
%We denote the largest, smallest and smallest positive eigenvalue of $A$ by $\lambda_{\max}(A),\lambda_{\min}(A)$, and
%$\lambda_{{\min}>0}(A)$.
%singular values are not used
%For a rectangular matrix $B\in \reals^{\dm_1\times \dm_2}$, we define $\sigma_i(B)$, $i=1,\dots, \min\{\dm_1,\dm_2\}
%,\sigma_{\max}(B), \sigma_{\min}(B)$ and $\sigma_{\min>0}(B)$ similarly to those for eigenvalues.
%

\paragraph{Inner product} We use the Euclidean inner product for vectors:
for $y$,$z \in \reals^\dm$, $\inprod{y}{z}=\sum_{i=1}^n y_iz_i$.
We use the trace inner product for matrices:
for $X$ and $Y\in \sym^{\dm}$ or $X$ and $Y\in\reals^{\dm_1\times \dm_2}$,
$\inprod{X}{Y}=\tr(X^\top Y) = \sum_{i,j}X_{ij}Y_{ij}$.

\paragraph{Transposes and adjoints} For a vector $v$ or a matrix $A$, $v^\top$ and $A^\top$ denote the transpose.
The adjoint map of a linear map $\Amap$ from $\sym^{\dm} \rightarrow \reals^\cons$ is defined
as the unique linear map $\Amap^*: \reals^m \rightarrow \sym^{\dm}$ such that for every $X\in \sym^{\dm},
y\in \reals^{\cons}$, $\inprod{\Amap X}{y}=\inprod{X}{\Amap^*y}$.

\paragraph{SDP Optimization} The notation $\xsol$ denotes a primal solution to \eqref{p} and $\ysol$ denotes a dual solution to $\eqref{d}$. Define the slack operator $Z: \reals^\dm \to \sym^\dm$
that maps a putative dual solution $y \in \reals^m$
to its associated slack matrix $Z(y) := C - \mathcal{A}^* y$. We omit the dependence on
$y$ if it is clear from the context.

\section{\newcontent{Simple} SDPs are generic}\label{sec: Regular low rank SDPs}
In this section, we first show that any psd matrix solves a \newcontent{simple} SDP.
We also demonstrate that for almost all $C$,
if SDP \eqref{p} satisfies primal Slater's condition and \newcontent{has a surjective constraint map}
and has a primal solution,
then it is primal \newcontent{simple.} We then show that
interesting SDPs, including MaxCut, OrthogonalCut, and ProductSDP (introduced in
\cref{sec: regularSDPs}),
are \newcontent{simple} for almost all $C$. Finally, we demonstrate numerically,
MaxCut SDP of many graphs are indeed simple.

\subsection{Any PSD matrix solves a \newcontent{simple} SDP }
Given any rank $\rsol$ positive semidefinite matrix $\xsol$,
we can construct a \newcontent{simple} SDP with $\xsol$ as its unique solution.

Write the eigenvalue decomposition of $\xsol$  as $\xsol= \sum_{i=1}^{\dm} \lambda_i v_iv_i^T=V\Lambda V^\top$.
Here the eigenvalues satisfy $\lambda_1\geq \cdots \geq \lambda_{\rsol}>\lambda_{\rsol+1}=\cdots =\lambda_n =0$,
and we define the diagonal matrix $\Lambda=\diag(\lambda_1,\dots,\lambda_{\rsol},0,\dots,0)$
and the orthonormal matrix $V= [v_1,\dots,v_n]\in \mathbb{R}^{n\times n}$.

We are now ready to construct the SDP and state our first theorem:
\begin{theorem}
	For any rank $\rsol$ positive semidefinite matrix $\xsol$ with
	eigenvalue decomposition $\xsol =\sum_{i=1}^{\dm} \lambda_i v_iv_i^T$,
	the SDP
	\beq \label{opt: exxyp}
	\ba{lll}
	\mbox{minimize} & \tr( X) & \\
	\mbox{subject to} & \tr(v_iv_i^\top X)=\lambda_i, & i=1,\dots, \rsol,\\
	& \tr(\frac{v_iv_j^\top+v_jv_i^\top}{2}X)= 0, & 1\leq i< j\leq \rsol,\\
	& X\succeq 0, & \\
	\ea
	\eeq
	with variable $X\in \sym^n$ is \newcontent{simple} and has $\xsol$ as its unique solution.
\end{theorem}
\begin{proof}
	Let us first write down the dual, with variables $y_{ij}$, $1\leq i \leq j \leq \rsol$:
	\beq \label{opt: exxyd}
	\ba{ll}
	\mbox  {maximize} & \sum_{i=1}^{\rsol}\lambda_iy_{ii}  \\
	\mbox{subject to}  & I- \sum_{i\leq j\leq \rsol} \frac{v_iv_j^\top +v_jv_i^\top}{2}y_{ij} \succeq 0.\\
	\ea
	\eeq
	We now verify each property required for \newcontent{simplicity} .
	
	\paragraph{\newcontent{Surjective constraint map}} The matrices $A_{ij} = \frac{v_iv_j^\top+v_jv_i^\top}{2}$, $1\leq i\leq  j\leq \rsol$ are
	orthogonal and hence they are linearly independent. As a result, $\Amap$ is sujective.
	\paragraph{Strong duality and strict complementarity}
	Define $\ysol$ with $({\ysol})_{ii}=1$ for
	$i=1\,,\ldots,\,\rsol$ and $({\ysol})_{ij}=0$ for $i \neq j \leq \rsol$. % $i < j \leq \rsol$
	To verify strong duality and the strict complementarity,
	we claim $\xsol$ and $\ysol$ are solutions to
	the primal SDP, \cref{opt: exxyp}, and dual SDP, \cref{opt: exxyd}, respectively.
	Indeed, it is easy to verify that $\xsol$ is primal feasible.
	Furthermore, by writing the slack matrix
	\[
	Z(\ysol)= I - \sum_{i\leq j\leq \rsol} \frac{v_iv_j^\top +v_jv_i^\top }{2}(\ysol) _{ij} =
	I -\sum_{i=1}^{\rsol} v_iv_i^\top\succeq 0,
	\]
	we see $\ysol$ is dual feasible and $Z(\ysol)$ has rank $n-\rsol$.
	Since the primal and dual objective match,
	$\tr(X) = \sum_{i=1}^{\rsol} \lambda_i = \sum_{i=1}^{\rsol} \lambda_i (\ysol)_{ii}$,
	we see $\xsol$ and $\ysol$ are a primal-dual optimal solution pair.
	Since $Z(\ysol)$ has rank $n-\rsol$ and $\rank(\xsol) =\rsol$,
	we see strict complementarity holds.
	
	\paragraph{Uniqueness}
	Suppose that $\ysol'$ solves the dual problem \eqref{opt: exxyd}.
	We will show that $\ysol' = \ysol$, and hence the dual has a unique solution.
	Using strong duality, we know $Z(\ysol')\xsol =0$.
	Moreover, $Z(\ysol')$ and $\xsol$ are psd.
	Hence $Z(\ysol')$ has rank at most $n-\rsol$.
	By the definition of $Z(\ysol')$,
	we see
	\begin{equation}
		\begin{aligned}\label{eqn: ZyprimeunderV}
			Z(\ysol')&=I -  \sum_{i\leq j\leq \rsol} \frac{v_iv_j^\top  +v_jv_i^\top }{2}{\ysol'} _{ij}\\
			&=V
			\begin{bmatrix}
				1 -(\ysol')_{11}& -\frac{(\ysol')_{12}}{2}&\dots&  -\frac{(\ysol')_{1\rsol}}{2} & 0 \\
				\vdots & \ddots&  & \vdots&  0\\
				-\frac{(\ysol')_{\rsol1}}{2}&\dots & & 1 - (\ysol')_{\rsol \rsol} & 0\\
				0 &\dots   &0 & &I_{n-\rsol}  \\
			\end{bmatrix}V^\top.
		\end{aligned}
	\end{equation}
	The lower right block of the inner matrix above is the identity $I_{n-\rsol}\in \sym ^{\dm -\rsol}$.
	Hence we see $Z(\ysol')$ has rank at least $n-\rsol$.
	Thus $Z(\ysol')$ must have rank exactly $n-\rsol$.
	This fact forces the upper left block of $Z(\ysol')$ in \eqref{eqn: ZyprimeunderV} to be $0$.
	Hence, we must have $\ysol' = \ysol$.
	
	To show the primal solution is unique,
	introduce the new variable $S \in \sym^\dm$
	so that $VSV^\top = X$.
	Using this change of variables in \eqref{opt: exxyp},
	we see $\xsol$ uniquely solves \eqref{opt: exxyp}
	if and only if $\Lambda \in \sym ^{\dm}$ uniquely solves
	\beq \label{opt: exxyprime}
	\ba{ll}
	\mbox{minimize} & \tr( S)\\
	\mbox{subject to} & S_{ii}=\lambda_i, \quad i=1,\dots, \rsol\\
	& S_{ij}= 0, \quad 1\leq i< j\leq \rsol,\\
	& S\succeq 0.\\
	\ea
	\eeq
	(Notice that $S=\Lambda$ is optimal for \cref{opt: exxyprime}, using the same
	argument we used to show the optimality of $\xsol$ for \cref{opt: exxyp} above.)
	Since the optimal value of \cref{opt: exxyprime} is $\tr(\Lambda)=\sum_{i=1}^{\rsol} \lambda_i$,
	from the constraints $S_{ii}=\lambda_i$ of \eqref{opt: exxyprime},
	we see that any feasible $S \succeq 0$ of \eqref{opt: exxyprime} has objective value
	$\geq \tr(\Lambda)$.
	To achieve optimality, we must have $S_{rr} = 0$ for $\rsol < r \leq \dm$.
	Now use the fact that $S \succeq 0$ to see $\Lambda$ is the unique solution.
\end{proof}
\subsection{Almost all cost matrices yield a primal \newcontent{simple} SDP}\label{sec: Primal regular SDP for generic C}
We establish the fact that \eqref{p} is primal \newcontent{simple} for almost all cost matrices $C$,
whenever the primal solution exists.

\begin{theorem}\label{thm: Primal regular SDP for generic C}
	Suppose \eqref{p} satisfies \newcontent{the surjective constraint map condition} and the primal
	Slater's condition: there is some $X_0 \in	\sym ^\dm$ such that $X_0 \succ 0$ and $\Amap (X_0)=b$.
	Then for almost all $C\in \sym ^\dm$,
	\eqref{p} is primal \newcontent{simple} as long as the primal solution exists.
\end{theorem}

\begin{proof} We utilize \cite[Corollary 3.5]{drusvyatskiy2011generic}: for a convex extended value function $f:\reals^\dm \rightarrow \reals \cup \{+\infty\}$, for almost all $v\in \reals ^{\dm}$, the perturbed function $f_v(x) = f(x)-v^\top x$
	admits at most one minimizer $x_v$ and satisfies $v\in \ri \left (\partial f(x_v)\right)$, the relative
	interior of $\partial f(x_v)$.
	
	To exploit this theorem, we set $v=-C\in \sym^{\dm}$ and take $f$ to be the function
	\begin{equation}\label{eq: fsdp}
		\begin{aligned}
			\indicator_{\{\Amap X =b\}} + \indicator_{\{X\succeq 0\}},
		\end{aligned}
	\end{equation}
	where $\indicator_{\mathcal C}(x)$ is the indicator function of a convex set $\mathcal C$:
	$0$ if $X\in \mathcal C$ and $+\infty$ otherwise.
	%
	%Note that all the examples in this section satisfies the Slater's condition:  there exists $X\succ 0$ satisfying $\Amap X = b,
	%X\succeq 0$.
	Using  \cite[Corollary 3.5]{drusvyatskiy2011generic},
	we see that for almost all $C$, the problem
	$\min (\indicator_{\{\Amap X =b\}} + \indicator_{\{X\succeq 0\}})(X) + \tr (CX)$ has at most one solution $X_C$, and
	\begin{equation}
		\begin{aligned}\label{eqn: subdifgenericC}
			-C&\in \ri\left( \partial \left( \indicator_{\{\Amap X =b\}} + \indicator_{\{X\succeq 0\}} \right)(X_C) \right)\\
			& \overset{(a)}{=} \ri\left( \partial \left( \indicator_{\{\Amap X =b\}} \right)(X_C)+ \partial \left(\indicator_{\{X\succeq 0\}} \right)(X_C) \right)\\
			& \overset{(b)}{=} \ri \left(\partial \left( \indicator_{\{\Amap X =b\}} \right)\right)(X_C)+ \ri \left(\partial \left(\indicator_{\{X\succeq 0\}} \right)(X_C) \right)\\
			& \overset{(c)}{=}-\{\Admap y\mid y\in \reals ^\cons \} -\{
			Z \mid Z\succeq 0,\, \ker(Z)=\range(X_C) \},
		\end{aligned}
	\end{equation}
	which implies that $C = \Admap y - Z$ for some slack matrix $Z\succeq 0$ satisfying $\ker(Z)=\range(X_C)$.
	Here step $(a)$ uses Slater's condition to apply the sum rule of the subdifferential.
	Step $(b)$ uses \cite[Proposition 1.3.6]{bertsekas2009convex}: the sum rule for the relative interior.
	Step $(c)$ uses basic sub-differential calculus. Hence, there is some $y$ and $Z$ such that
	$Z=C-\Amap^*y \succeq 0$, $\ker(Z)=\range(X_C)$, and
	\[
	\rank (Z) + \rank (X_C)=n \quad \textrm{and} \quad % \quad \ker(Z) =\range(X_C) \implies
	\tr (Z X_C)=0.
	\]
	Hence $y$ is dual optimal and strict complementarity holds.
	%
	%	The uniqueness of dual for \ref{MaxCut}, \ref{ocut}, and \ref{sdpproduct}  can be verified as the proof of Theorem \ref{thm: failBMocut}. Hence the SDP in this sections are all regular.
\end{proof}

\subsection{MaxCut-type SDP are \newcontent{simple} for almost all $C$}\label{sec: regularSDPs}
In this section, we introduce three classes of SDPs that generalize the
SDP relaxation of the MaxCut problem \cite{goemans1995improved},
with applications in statistical
signal recovery, optics, and subproblems of important algorithms.
We show in \cref{cor: sdpregular}  that
they are \newcontent{simple} for almost all $C$ based on \cref{thm: Primal regular SDP for generic C}.

\paragraph{MaxCut}
We call an SDP a \ref{MaxCut}-type SDP if it is of the form
\beq\label{MaxCut}\tag{MaxCut}
\ba{ll}
\mbox{minimize} & \tr(CX)\\
\mbox{subject to} & \diag( X) =  1
\quad\text{and}\quad
X \succeq 0. \\
\ea
\eeq
Here we do not require the cost matrix $C$ to be a negative Laplacian matrix.

\ref{MaxCut}-type SDP can be used to find approximations
of the maximum weight cut in a graph \cite{goemans1995improved},
to recover an object of interest from optical measurements~\cite{waldspurger2015phase},
and to identify the cluster corresponding to each node in the stochastic block model \cite{bandeira2018random}.

\paragraph{Orthogonal cut}
For any $M\in \reals^{S\dg\times S\dg},s\leq S$,
we denote by $\block_s(M)$ the $s$-th diagonal $\dg \times \dg$ block of $M$.
An \ref{ocut}-type problem has decision
variable $X\in \sym^{S\dg}$ for some integer $S>0$ and $\dg=1,2$, or $3$,
and is of the form
\beq\label{ocut}\tag{OrthogonalCut}
\ba{ll}
\mbox{minimize} & \tr(CX)\\
\mbox{subject to} &  \block_s(X) =  I_\dg, \quad s=1,\dots,S,\\
&
X \succeq 0. \\
\ea
\eeq
Note that when $\dg=1$, \eqref{ocut} reduces to \eqref{MaxCut},
with $m = \frac{S\dg(\dg+1)}{2}$ constraints.

The \ref{ocut}-type SDP generalizes the \ref{MaxCut}-type SDP,
and appears in sensor network localization \cite{cucuringu2012sensor}
and ranking problems \cite{cucuringu2016sync}.

\paragraph{ProductSDP: optimization over a product of spheres}
Finally, we introduce \eqref{sdpproduct},
an SDP relaxation of a quadratic program over a product of spheres.
Let $D$ be a positive integer and
let $S_1,\dots,S_m$ be a partition of the set $[D]:=\{1,\dots,D\}$:
$S_i\cap S_j =\emptyset$ for all $i\not=j$, and $\cup_{i=1}^m S_i = [D]$.
A \ref{sdpproduct}-type problem, with decision variable $X\in \sym^{D}$,
takes the form
\beq\label{sdpproduct}\tag{ProductSDP}
\ba{ll}
\mbox{minimize} & \tr(CX)\\
\mbox{subject to} & \sum_{k\in S_i} X_{kk}=1, \quad i=1,\dots,m,\\
&	X \succeq 0. \\
\ea
\eeq
Note that when the cardinality of each $S_i$ is one, \eqref{sdpproduct} reduces to \eqref{MaxCut}.

To explain the name of this SDP, suppose $x_i \in \mathbb{R}^{|S_i|}$ for each $i=1,\ldots,m$.
The constraint $\inprod{x_i}{x_i}=1$ ensures that $x_i$ is on
the sphere $\mathcal{S}^{|S_i|-1}$ in $\mathbb{R}^{|S_i|}$.
Now stack the variables $x_i$ for $i=1,\ldots,m$
as a vector $x\in \mathbb{R}^n$.
The SDP \eqref{sdpproduct} is a relaxation of the quadratic program
%	The Burer-Monteiro formulation in this case
%	is simply
\beq\label{bmproduct}%\tag{BM-Product}
\ba{ll}
\mbox{minimize} & \tr(Cxx^\top)\\
\mbox{subject to} & x\in \prod_{i=1}^{m}\mathcal{S}^{|S_i|-1}
\ea
\eeq
with $xx^\top$ replaced by $X$.
Problems of this form can appear as trust-region subproblems,
\eg \cite[Section 5.3]{boumal2018deterministic}.

Having defined these three classes of SDP,
we show all of these problems are almost always simple.

\begin{corollary}\label{cor: sdpregular}
	\newcontent{The \ref{MaxCut}-type, \ref{ocut}-type, and \ref{sdpproduct}-type SDPs} are simple	for almost any cost matrix $C$.
\end{corollary}
\begin{proof}
	We first check dual uniqueness and \newcontent{verify the constraint map is surjective}.
	Then we verify primal \newcontent{simplicity}  to conclude that these three classes of SDP are simple.
	
	\paragraph{Dual uniqueness and \newcontent{surjective constraint map}}
	First, note the property \newcontent{surjective constraint map} follows directly from the uniqueness
	of the dual solution.
	We show dual uniqueness by contradiction:
	if the dual is not unique,
	there is some $\Delta y$ such that $\Amap^*(\Delta y)=0$ and
	$\ysol+\alpha \Delta y$ for some $\alpha\in \mathbb{R}$ is still optimal.
	Using \cite[Proposition 9]{waldspurger2018rank}, we know there is no nonzero $y$ such that
	\[
	\Admap(y)\xsol =0.
	\]
	It is then immediate the dual is unique by noting $Z(y)\xsol =0$ for any dual optimal $y$.
	\footnote{\newcontent{In fact,
			primal nondegeneracy [Definition 5]\cite{alizadeh1997complementarity}, a stronger
			condition than dual uniqueness,
			always holds for \ref{MaxCut} even when strict complementarity fails.
			To prove this fact, one can check
			the definition directly, or check \cite[Assumption 1.1a]{boumal2018deterministic} and use \cite[Proposition 6.6]{boumal2018deterministic}.}}
	
	\paragraph{Primal \newcontent{simplicity}} The primal solution exists because the feasible region of each class is compact and nonempty.
	Slater's condition for these three classes of SDP can be easily verified using a well-chosen diagonal matrix. Hence Theorem \ref{thm: Primal regular SDP for generic C} asserts these three classes are primal \newcontent{simple} for almost all $C$.
\end{proof}

\subsection{Numerical verification for real-world SDP}
In this section, we numerically verify that the MaxCut problems \eqref{MaxCut}
corresponding to several graphs are simple.
In particular, we use the Gset graphs G1 to G20 \cite{Gset};
in the MaxCut relaxation, the cost matrix $C$ is the negative graph Laplacian.
Each graph has $n=800$ vertices, so
the MaxCut SDP \eqref{MaxCut} has a decision variable $X$ of size $800\times 800$.

To verify strict complementarity, we must compute the rank of the primal and
dual solution $\xsol$ and $\zsol$, $r_p$ and $r_d$, and see whether $r_p+r_d =\dm$.

To verify uniqueness of the primal solution,
define a matrix $U\in \reals^{\dm \times(n- r_d)}$ whose columns form an orthonormal basis
for the null space of $Z$.
Define the linear operator $\Amap_{\zsoloy}:\sym^{\dm-r_d}\rightarrow \reals ^\dm$,
$\Amap_U(S)=USU^\top$, where $\Amap = \diag$.
According to \cite{alizadeh1997complementarity}[Theorem 9 \& 10],
the primal solution $\xsol$ is unique if the smallest singular value $\sigma_{\min}(\Amap_{\zsoloy })
\defn\min_{\fronorm{S}=1}\twonorm{\Amap_{\zsoloy}(S)}$ is nonzero.

To verify uniqueness of the dual solution,
define a matrix $V_1\in \reals^{\dm \times {r_p}}$ whose columns form
an orthonormal basis for the column space of $X$
and $V_2 \in \reals^{\dm \times (\dm -r_p)}$ whose columns form a basis for the null space of $X$.
Define the matrix $\Amap^*_{\xsol} \in \reals ^{\dm r_p\times \dm }$ where
the $k$-th column of $\Amap^*_{\xsol }$ is
$[\Amap^*_{\xsol}]_{\cdot k} =  \mathbf{vec}\left(  \begin{bmatrix}
	V_1^\top e_ke_k^\top V_1 \\
	V_2^\top e_k e_k^\top  V_1
\end{bmatrix}\right)$ for $k=1,\dots,\dm$\footnote{Here $e_i$ is the $i$-th standard basis vector in $\reals^\dm$
	and $\mathbf{vec}$ stacks the columns of a matrix.}.
Then according to \cite[Theorem 6 \& 7]{alizadeh1997complementarity}, the dual is unique if
the smallest singular value of $\Amap^*_{\xsol}$ is nonzero.

Numerically, we obtain $\xsol$ and $\zsol$ using the MOSEK solver \cite{mosek2010mosek}.
We estimate the rank by the number of eigenvalues larger than $10^{-6}$,
and denote the smallest eigenvalue larger than $10^{-6}$ as $\lambda_{\min>0}(\xsol)$ and
$\lambda_{\min>0}(\zsol)$ respectively. We compute their condition numbers defined as
$\kappa_{\xsol} := \frac{\lambda_1(\xsol)}{\lambda_{\min>0}(\xsol)}$
and
$\kappa_{\zsoloy}: = \frac{\lambda_{1}(\zsoloy)}{\lambda_{\min>0}(\zsoloy)}$.
We compute the condition numbers of $\Amap_{\xsol}$ and $\Amap_{\zsoloy}^*$ defined as
$\kappa\left(\Amap _{\zsoloy}\right):=\frac{\opnorm{\Amap_{\zsoloy}}}{\sigma_{\min}(\Amap_{\zsoloy })}$
and
$\kappa\left(\Amap^* _{\xsol}\right):=\frac{\sigma_1(\Amap^*_{\xsol})}{\sigma_{\dm}\left(\Amap^* _{\xsol} \right)}$.
The results are reported in Table \ref{table: numericalVerficationMaxCut}.
As can be seen, \newcontent{simplicity}  is indeed satisfied for every MaxCut problem from G1 to G20.
For graph G11, the condition number $\zsoloy$ is about $3\times 10^{5}$ and its $\lambda_{\min>0}(\zsol)$
is actually only $10^{-5}$ (not shown here) meaning that strict complementarity holds in a very weak sense.
\begin{table}
	\begin{tabular}{|cccccccc|}
		\hline Graph& $\dm$ &
		$\rank(\xsol)$ & $\rank(\zsoloy)$ %& $\lambda_{\min>0}(\xsol)$& $\lambda_{\min>0}(\zsoloy)$
		&$\kappa_{\xsol}$ & $\kappa_{\zsoloy}$
		%& $\sigma_{\min}(\Amap_{\zsoloy})$ & $\sigma_{\min}(\Amap^*_{\xsol})$
		& $\kappa(\Amap_{\zsoloy})$ & $\kappa(\Amap^*_{\xsol})$\\
		\hline
		\hline
		G1& 800 &
		13 &  787 %& 9.356  & 0.01885
		& 13.99 & 3269.0
		% & 0.03205 & 0.09422
		& 4.083 & 2.026\\
		%\hline
		G2& 800
		& 13
		& 787  %& 10.67 & 0.02229
		& 11.65 & 2770.0
		%& 0.03175 & 0.09617
		& 4.123 & 1.901\\
		%\hline
		G3&
		800
		& 14 & 786
		%& 0.7711 & 0.03885
		& 187.4 & 1590.0
		%& 0.03127 & 0.09472
		& 4.413 & 2.062\\
		%\hline
		G4&
		800&
		14 & 786
		%& 1.89 & 0.09133
		& 78.68 & 678.0
		%& 0.03064 & 0.09225
		& 4.527 & 2.419\\
		%\hline
		G5&
		800
		& 12 & 788
		% & 8.15 & 0.02741
		& 18.24 & 2258.0
		% & 0.03314 & 0.0922
		& 3.841 & 2.055\\
		%\hline
		G6&
		800
		& 13 & 787
		%& 2.907 & 0.02319
		& 50.17 & 1206.0
		%& 0.03183 & 0.0938
		& 4.149 & 1.935\\
		%\hline
		G7&
		800
		& 12 & 788
		% & 10.8 & 0.001117
		& 12.91 & 25060.0
		%& 0.03299 & 0.09328
		& 3.84 & 2.085\\
		%\hline
		G8&
		800
		& 12 & 788
		% & 2.69 & 0.05683
		& 50.32 & 496.7
		%& 0.03326 & 0.0928
		& 3.84 & 2.331\\
		%\hline
		G9&
		800&
		12 & 788
		%& 14.12 & 0.04526
		& 10.29 & 619.7
		% & 0.03291 & 0.09157
		& 3.845 & 2.086\\
		%\hline
		G10&
		800
		& 12 & 788
		% & 11.01 & 0.02744
		& 11.53 & 1008.0
		%& 0.0334 & 0.09424
		& 3.777 & 1.929\\
		% \hline
		G11&
		800
		& 6 & 794
		% & 22.17 & 2.076e-5
		& 10.14 & $3.404\times 10^5$
		% & 0.01206 & 0.06716
		& 9.478 & 2.619\\
		% \hline
		G12&
		800
		& 8 & 792
		% & 4.911 & 0.00015
		& 48.88 & 48370.0
		%& 0.01093 & 0.06682
		& 9.876 & 2.445\\
		%\hline
		G13&
		800
		& 8 & 792 %& 3.847 & 0.001399
		& 56.67 & 5221.0
		% & 0.01489 & 0.07064
		& 7.161 & 2.177\\
		%\hline
		G14&
		800
		& 13 & 787
		% & 7.098 & 0.01742
		& 18.61 & 2517.0
		%& 0.02926 & 0.09352
		& 4.525 & 2.222\\
		%\hline
		G15&
		800
		& 13 & 787
		% & 4.679 & 0.006313
		& 33.7 & 7516.0
		%& 0.02938 & 0.08822
		& 4.523 & 2.237\\
		%\hline
		G16&
		800
		& 14 & 786
		% & 0.4796 & 0.0177
		& 270.8 & 2443.0
		%& 0.02802 & 0.1009
		& 4.918 & 2.059\\
		%\hline
		G17&
		800
		& 13 & 787
		% & 0.761 & 0.01883
		& 177.2 & 2323.0
		% & 0.02922 & 0.09737
		& 4.521 & 2.199\\
		% \hline
		G18&
		800
		& 10 & 790
		% & 10.77 & 0.008312
		& 13.42 & 6182.0
		%& 0.02944 & 0.08518
		& 3.932 & 1.915\\
		%\hline
		G19&
		800
		& 9 & 791
		%& 16.24 & 0.005826
		& 10.89 & 10580.0
		%& 0.03164 & 0.0779
		& 3.51 & 2.149\\
		%\hline
		G20&
		800
		& 9 & 791
		% & 1.063 & 0.01418
		& 172.9 & 3586.0
		%& 0.03166 & 0.08113
		& 3.502 & 2.246\\
		\hline
	\end{tabular}
	\caption{Summary statistics of MaxCut problem on Gset graphs verify \newcontent{simplicity} .
		Strict complementarity holds, as numerically $\rank(\xsol) + \rank(\zsoloy)$
		is equal to the dimension $n=800$ for every problem.
		% the primal and dual decision variables can have large condition numbers,
		The small condition numbers of the linear maps $\Amap_{\zsoloy}$
		and $\Amap^*_{\xsol}$ verify primal and dual uniqueness, respectively.}
	\label{table: numericalVerficationMaxCut}
\end{table}

\section{Burer-Monteiro may fail for \newcontent{simple} SDPs}\label{sec: bmfail}
In this section, we show that the \eqref{BM} formulation of \eqref{p}
admits second order stationary points that are not globally optimal
even for \newcontent{simple} SDPs with low rank ($1$ or $2$ or $3$) solutions.

Recall from the introduction the \emph{Burer and Monteiro approach} (BM approach)
to semidefinite programming, which replaces the SDP \eqref{p} by
the following nonlinear optimization problem with decision variable $F\in \reals^{\dm \times r}$:
\begin{equation}\tag{BM}\label{bm}
	\ba{ll}
	\mbox{minimize} & \tr( CFF^\top )=:f(F)\\
	\mbox{subject to} & \mathcal{A}(FF^*) = b.
	\ea
\end{equation}
This problem is in general nonconvex.

Nonlinear optimization solvers such as Riemannian trust regions \cite{boumal2018global}
can guarantee that they find a second order stationary
point (SOSP) of such a problem, but cannot guarantee
(or even check) that they have found a global solution.
When the constraint set is a
manifold, as it is for all the examples discussed in the previous section,
a putative solution $F$ is second order stationary
if its Riemannian gradient is $0$ and its Riemannian Hessian is positive semidefinite.
See Appendix \ref{sec: appendixBMstationary} for further discussion.

Hence we can guarantee that the BM approach finds the global optimum
if we can prove that all SOSPs are globally optimal.
The following definition serves as a useful shorthard as we understand when this condition holds.
%With SOSP thus defined, we restate the concept of failure of \eqref{bm} formally here:
\begin{definition}
	We say the BM approach \emph{succeeds} for an SDP \eqref{p}
	if every SOSP $F$ of \eqref{bm} is globally optimal,
	and hence $X=FF^\top$ is optimal for \eqref{p}.
	Conversely, we say the BM approach \emph{fails}
	if \eqref{bm} has any SOSP that is not global optimal.
\end{definition}
% This definition with a focus on provable guarantees.
Note that as a practical matter, a nonlinear solver for \eqref{bm} might
produce a globally optimal SOSP even for a problem that admits non-optimal SOSPs.

Recall from the introduction that for almost all $C$, when $\frac{r(r+1)}{2}>m$,
any SOSP of \eqref{bm} is globally optimal \cite{boumal2018deterministic}.
On the other hand, building on results by \cite{waldspurger2018rank},
we will demonstrate
a positive measure set of \newcontent{simple} SDP of each
of the three classes described in \Cref{sec: regularSDPs}
for which BM fails whenever $\frac{r(r+1)}{2}+r\leq m$.

\subsection{Examples: MaxCut, OrthogonalCut, and ProductSDP}
Let us first recall the \eqref{MaxCut}  SDP we described in \cref{sec: regularSDPs}:
\beq \tag{MaxCut}
\ba{ll}
\mbox{minimize} & \tr(CX)\\
\mbox{subject to} & \diag( X) =  1
\quad\text{and}\quad
X \succeq 0. \\
\ea
\eeq
As demonstrated in \cite[Corollary 1]{waldspurger2018rank}, if
\[
\frac{r(r+1)}{2}+r >n,
\]
then for almost all $C$, any SOSP $F$ of the BM formulation \eqref{BM} of \eqref{MaxCut} is global optimal.
Hence the matrix $FF^\top$ solves \eqref{MaxCut}.
However in \cite[Corollary 1]{waldspurger2018rank}, the authors show that if
\[
\frac{r(r+1)}{2}+r \leq n,
\]
then there is a positive measure set of the cost matrix $C$ for which \eqref{MaxCut} has a unique rank $1$ solution but the BM approach fails.

Are these SDP particularly nasty?
On the contrary! Our contribution, stated in the following theorem, is to show that
these SDPs are \newcontent{simple}.
We also generalize these results to \eqref{ocut} and \eqref{sdpproduct}.

\begin{theorem}\label{thm: failBM} Fix a positive integer $r$. If
	\[
	\frac{r(r+1)}{2}+r \leq n,
	\]
	then there is a set of cost matrices $C$ with positive measure
	for which \eqref{MaxCut} admits a unique rank $1$ solution and is \newcontent{simple},
	but the BM approach fails.
	
	The same result holds for \eqref{sdpproduct} if \[
	\frac{r(r+1)}{2}+r \leq m.
	\]
	For \eqref{ocut}, the same result holds, except that the solution has rank $\dg$, if
	\[
	\frac{r(r+1)}{2}+r\dg \leq m = \frac{S\dg(\dg+1)}{2}.
	\]
\end{theorem}
\begin{proof}
	The proofs of dual uniqueness and  \newcontent{the surjective constraint map property} are the same as in the proof of \cref{cor: sdpregular}.
	We next verify the failure of BM, and primal \newcontent{simplicity} .
	\paragraph{Failure of BM, and Primal \newcontent{simplicity}}
	Waldspurger and Waters show that there is a positve measure set of cost matrices $C$
	for which \eqref{MaxCut} satisfies:
	(1) strong duality \cite[Proposition 4]{waldspurger2018rank},
	(2) uniqueness of a primal solution $\xsol$ with rank $1$ \cite[Corollary 2]{waldspurger2018rank},
	(3) strict complementarity for a dual solution $\ysol$ \cite[Lemma 2, Lemma 9 and $\xsol Z(\ysol)=0$]{waldspurger2018rank},
	(4) the BM approach fails \cite[Corollary 1]{waldspurger2018rank}.
	Together with dual uniqueness and \newcontent{the surjective constraint map property},
	these results verify the theorem statement for \eqref{MaxCut}.
	
	\paragraph{OrthogonalCut and ProductSDP} The proof for the other two SDPs
	follows exactly the same argument as above,
	using \cite[Corollary 2]{waldspurger2018rank} for \eqref{ocut}
	and \cite[Corollary 3]{waldspurger2018rank} for \eqref{sdpproduct}.
\end{proof}

\section{Noisy SDPs are simple} \label{sec: weakbm}
In section \ref{sec: Regular low rank SDPs},
we saw that many interesting SDPs are \newcontent{simple} for almost any cost matrix $C$.
In this section, we show that the (very structured)	cost matrices
that appear in certain statistical problems also yield \newcontent{simple} SDPs.
In these problems, the objective measures agreement with observations of a ground-truth object,
while the constraints restrict the complexity of the solution.
Importantly, \newcontent{simplicity}  of these problems
guarantees that the solution of the SDP recovers the ground truth.

More precisely, we consider the SDP relaxations of the following statistical problems:
\begin{itemize}
	\item $\integers_2$ Synchronization
	\item Stochastic Block Model
\end{itemize}
We show that these SDP relaxations are \newcontent{simple} with high probability.

We also demonstrate a strong advantage to solving the original SDP rather than
using the BM approach (when applicable): these SDPs can provably recover the ground truth
under much higher noise
\newcontent{than the noise level (provably) allowable using the BM approach.}

\subsection{$\integers_2$ Synchronization} \label{sec: Z2}
Consider a binary vector $z\in \{\pm 1\}^\dm$.
The $\integers_2$ synchronization problem is to
to recover the vector $z$ up to a sign
from the observations $Y = zz^\top + \gamma W$,
where $W$ is symmetric with iid  standard normal upper diagonal entries, and $0$ diagonal
entries. The value $\gamma$ is the noise level.
The SDP proposed in the literature with decision variable $X\in \sym^\dm$ is
\beq\label{Z_2sync}\tag{\text{$\integers_2$ Sync}}
\ba{ll}
\mbox{minimize} & \tr(-YX)\\
\mbox{subject to} & \diag(X) =  1
\quad\text{and}\quad
X\succeq 0. \\
\ea
\eeq
The corresponding Burer-Monteiro formulation with variable $F\in \reals^\dm \times r$ is
\beq\label{bmZ_2}\tag{\text{BM $\integers_2$ Sync}}
\ba{ll}
\mbox{minimize} & \tr(-YFF^\top)\\
\mbox{subject to} & \diag( FF^\top) =  1. \\
\ea
\eeq

It is intuitive that the problem is more challenging as the
noise level $\gamma$ increases.
For \eqref{Z_2sync}, if the noise level satisfies
$\gamma \leq \sqrt{\frac{\dm}{(2+\epsilon)\log \dm}}$ for some numerical constant $\epsilon>0$,
it admits $zz^\top$ as its unique solution with high probability \cite[Proposition 3.6]{bandeira2018random}.
But for \eqref{bmZ_2} with $r=2$, the best known theoretical results state that
the noise level $\gamma$ must be less than $c\dm^{\frac{1}{6}}$ for some small numerical constant $c>0$
to ensure the BM formulation succeeds,
\ie all second order stationary points $F$ satisfy $FF^\top =zz^\top$\cite{bandeira2016low}.
The gap between $\gamma =\sqrt{\frac{\dm}{(2+\epsilon)\log \dm}}$  and
$\mathcal{O}(\dm^{\frac{1}{6}})$ is polynomially large.

We now prove \ref{Z_2sync} is \newcontent{simple} whenever $\gamma \leq \sqrt{\frac{\dm}{(2+\epsilon)\log \dm}}$.
The uniqueness of the primal is proved in \cite[Proposition 3.6]{bandeira2018random}.
The dual optimal solution proposed in \cite[Proposition 3.6]{bandeira2018random} is
%\footnote{See \cite[Section 4.3]{bandeira2017tightness} why such construction make senses.} p
\[
\ysol = -\ddiag(Yzz^\top), \quad \text{and} \quad \zsoloy= -Y-(\diag(\ysol))= \ddiag(Yzz^\top)-Y,
\]
where $\ddiag:\sym ^\dm \rightarrow \reals^\dm $ is the adjoint operator of $\diag: \reals^\dm \rightarrow \sym^\dm $. Note that
$\zsoloy z = \ddiag(Yzz^\top)z-Yz=0$ using the fact that $z\in \{\pm 1\}^\dm$.
Using the proof of
\cite[Proposition 3.6, proof on pp356]{bandeira2018random} and $\gamma \leq \sqrt{\frac{\dm}{(2+\epsilon)\log \dm}}$,
we find that with high probability, there is a numerical constant $c\in (0,1)$ such that
\begin{equation}
	\begin{aligned}\label{eq: z2dualoptimallambda2}
		\lambda_{n-1}(\zsoloy)\geq cn.
	\end{aligned}
\end{equation}
We see $\zsoloy \succeq 0$ and is optimal as $\zsoloy zz^\top = 0$.
Moreover, strict complementarity is satisfied, as $\lambda_{n-1}(\zsoloy)>0$.
\newcontent{Surjectivity of the constraint map} and the uniqueness of the dual can be verified in the same way
as in the proof of Theorem \ref{thm: failBM}.
We summarize our findings as the following theorem.
\begin{thm}
	For the $\integers_2$ synchronization problem,
	if the noise level $\gamma <\sqrt{\frac{\dm}{(2+\epsilon)\log \dm}}$ for some
	numerical constant $\epsilon>0$, then with high probability
	the SDP \eqref{Z_2sync} is \newcontent{simple} with primal solution $zz^\top$.
	Moreover, the dual solution satisfies $\lambda_{n-1}(\zsoloy)>cn$ for some numerical constant $c\in(0,1)$.
\end{thm}

\subsection{Stochastic Block Model}
The stochastic block model (SBM) is structurally quite similar to $\integers_2$ synchronization.
The SBM posits that we observe the edges and vertices of a graph $\mathcal{G}$
with $n$ vertices that are split into two clusters according to a binary
membership vector $z\in\{-1,1\}^\dm$.
For each pair of vertices $(i,j) \in [n] \times [n]$ with $i\not=j$,
the undirected edge $(i,j)$ is formed with probability $p\in [0,1]$ if vertices $i$ and $j$
are in the same cluster ($z_i = z_j$) and with probability $0 \leq q < p$ otherwise. %if $z_i z_j=1$.
The goal is to recover the cluster membership vector $z$.
For simplicity, we further assume that $n$ is \emph{even} and that the clusters are balanced:
$n/2$ entries of $z$ are $+1$ and $n/2$ are $-1$.
Let $A$ be the adjacency matrix of $\mathcal{G}$ with diagonal entries set to be $\frac{p-q}{2}$.
The SDP proposed to recover $z$ by \cite{bandeira2016low}, with variable $X$, is
\beq\label{SBM}\tag{SBM}
\ba{ll}
\mbox{maximize} & \tr((A-\frac{p+q}{2}J)X)\\
\mbox{subject to} & \diag(X) =  1
\quad\text{and}\quad
X\succeq 0. \\
\ea
\eeq
where the matrix $J=\ones\ones^\top -I \in \sym^{\dm}$.
The corresponding Burer-Monteiro formulation with variable $F\in \reals^{\dm \times r}$ is
\beq\label{bmSBM}\tag{BM SBM}
\ba{ll}
\mbox{minimize} & \tr((A-\frac{p+q}{2}J)FF^\top)\\
\mbox{subject to} & \diag( FF^\top) =  1. \\
\ea
\eeq
(There are other SDP formulations for SBM which make weaker assumptions;
% do not require both $(p,q)$ and equal size clusters
see \cite{bandeira2018random}.
However, there are no guarantees for the corresponding Burer-Monteiro relaxations.)

To see the relation between \eqref{Z_2sync} and \eqref{SBM},
we note the cost matrix $A-\frac{p+q}{2}J$ can be decomposed as
\[
A-\frac{p+q}{2}J = \frac{p-q}{2}zz^\top + E,
\]
where the error matrix $E$ has zero diagonal, expectation $0$, and satisfies that for $z_i=z_j$, $i\not=j$,
\[
E_{ij} =\begin{cases}
	1-p & \text{with probability}\, p\\
	-p &\text{with probability}\, 1-p
\end{cases}
\]
and for $z_i\not=z_j$,
\[
E_{ij} =\begin{cases}
	1-q & \text{with probability}\, q\\
	-q &\text{with probability}\, 1-q.
\end{cases}
\]
We may rescale the cost matrix $A-\frac{p+q}{2}J$ by $\frac{2}{p-q}$ to form
\[
\truA = \frac{2}{p-q}(A-\frac{p+q}{2}J)= zz^\top+\frac{2}{p-q}E,
\]
which has the same form as the observation matrix $Y = zz^\top+\gamma W$ in Section \ref{sec: Z2}.

To establish the fact that \eqref{SBM} is simple, let us work with the following form of SDP
whose cost matrix is the rescaled version $\truA$:
\beq\label{SBMrsc}\tag{\text{$\truA$-SBM}}
\ba{ll}
\mbox{maximize} & \tr(\truA X)\\
\mbox{subject to} & \diag(X) =  1
\quad\text{and}\quad
X\succeq 0. \\
\ea
\eeq
Clearly \eqref{SBM} is \newcontent{simple} if and only if \eqref{SBMrsc} is \newcontent{simple}.
The dual certificate we construct here is
\[
\ysol = -\ddiag(\truA zz^\top), \quad \text{and} \quad \zsoloy= -\truA-(\diag(\ysol))= \ddiag(\truA zz^\top)-\truA,
\]

Using \cite[Lemma 11 and its proof]{bandeira2016low}, $\zsol$ is dual optimal, certifies
$zz^\top$ as the unique solution of \eqref{SBMrsc} and satisfies
$\lambda_2(\zsol)>cn$ (for some small but universal constant $c$) if
(for some large but universal constant $C$)
\[
\opnorm{\frac{2}{\sqrt{n}(p-q)}E}\leq \sqrt{\frac{\dm}{C\log\dm}}, \quad \text{and} \infnorm{\frac{2}{\sqrt{n}(p-q)}Ez}\leq \sqrt{\frac{n}{C}}.
\]
Here $\infnorm{\frac{2}{\sqrt{n}(p-q)}Ez}$ is the largest entry in absolute value of $\frac{2}{\sqrt{n}(p-q)}Ez$. 
Using \cite[Lemma 18, 19]{bandeira2016low}, the inequality of the previous display holds
if the signal strength satisfies
\[
\lambda(p,q) \defn \frac{p-q}{\sqrt{2(p+q)}}\sqrt{n}\geq C\sqrt{\log \dm}
\]
for some large universal constant $C$.

\cite[Theorem 6]{bandeira2016low} also states conditions under which \eqref{bmSBM}
provably succeeds. These conditions require $\lambda(p,q)\geq Cn^{1/3}$.
This requirement is polynomially larger than the $\lambda(p,q)$ that guarantees \newcontent{simplicity}  of \eqref{SBM}.
We summarize our findings as the following theorem.
\begin{thm}
	For the SBM problem, if the signal strength satisfies that $\lambda(p,q)= \frac{p-q}{\sqrt{2(p+q)}}\sqrt{n}\geq C\sqrt{\log \dm}$ for some
	numerical constant $C>0$, then with high probability, the SDP \eqref{SBM} is \newcontent{simple} with primal solution $zz^\top$ and the dual solution of \eqref{SBMrsc}
	satisfies $\lambda_{n-1}(\zsoloy)>cn$ for some numerical constant $c\in(0,1)$.
\end{thm}

\section{Primal \newcontent{simple} SDP: Matrix Completion}\label{sec: mc}
We have seen many \newcontent{simple} SDPs in previous sections. In this section,
we demonstrate that the matrix completion SDP,
a celebrated method for data imputation,
is not \newcontent{simple} but only primal \newcontent{simple}.

The matrix completion problem seeks to recover
a rank $\trur$ matrix $\trux \in \reals^{n_1 \times n_2}$ from
a few entrywise observations $(\trux)_{ij}, (i,j)\in \Omega$,
where $\Omega \subset [\dm_1]\times [\dm_2]$ is an index set of the observed entries.
Define the projection operator $\projomega {}:\reals^{\dm_1\times \dm_2} \rightarrow \reals^{\dm_1 \times \dm_2}$
as $[\projomega{}(A)]_{ij}= A_{ij}$ if $(i,j)\in \Omega$ and $0$ otherwise.

One popular recovery method for matrix completion, Nuclear Norm Minimization (NNM) \cite{candes2009exact},
imputes the missing entries by solving the SDP
\beq\label{mc} \tag{Matrix-Completion}
\ba{ll}
\mbox{minimize} & \|X\|_*\\
\mbox{subject to} & \projomega{}(X) = \projomega{}(\trux).\\
\ea
\eeq

A standard result in this literature \cite[Lemma 2]{fazel2002matrix}
represents the nuclear norm
by semidefinite-representable constraints on a lifted matrix
$\begin{bmatrix}
	X_1& X\\ X^\top & X_2
\end{bmatrix}$:
\begin{equation}
	\begin{aligned}\label{eq: nucsdpequip}
		\|X\|_*\leq t \iff \exists X_1,X_2 \;\text{such that}\; \begin{bmatrix}
			X_1& X\\ X^\top & X_2
		\end{bmatrix}\succeq 0, \tr(X_1)+\tr(X_2)\leq 2t.
	\end{aligned}
\end{equation}
Hence \eqref{mc} can be reformulated as
\beq
\ba{ll}\label{mclarge} \tag{SDP Matrix-Completion}
\mbox{minimize} & \tr(W_1)+\tr(W_2) \\
\mbox{subject to} & X_{ij} = (\trux)_{ij},\, (i,j)\in \Omega \\
& \begin{bmatrix}
	W_1 & X \\ X^\top & W_2
\end{bmatrix}\succeq 0,
\ea
\eeq
where $W_1,W_2,X$ are the decision variables. In particular, if $\xsol=\begin{bmatrix}
	W_1 & X\\
	X^\top & W_2
\end{bmatrix}$ solves \eqref{mclarge}, then $X$ solves \eqref{mc} using \eqref{eq: nucsdpequip}.

As is standard in this literature,
we measure the difficulty of the matrix completion problem by the incoherence $\mu$,
which we now define.
Let $\trux = U\Sigma V^\top$ be the SVD of $\trux$ with $U\in \reals^{\dm_1\times \trur},V\in \reals^{\dm_2\times \trur}$ having orthonormal columns and the diagonal matrix $\Sigma \in \reals^{\trur\times \trur}$ having positive entries on the diagonal. The incoherence $\inco$ is the smallest number that satisfies
\begin{equation}\label{eq: incoehrence}
	\begin{aligned}
		\max_{1\leq i\leq \dm_1} \twonorm{e_i^\top U}\leq \sqrt{\frac{\inco \trur}{\dm_1}}
		\qquad \text{and} \qquad
		\max_{1\leq i\leq \dm_2} \twonorm{e_i^\top V}\leq \sqrt{\frac{\inco \trur}{\dm_2}}.
	\end{aligned}
\end{equation}

If each entry of $\trux$ is observed independently with probability $p$ and $\trux$ is $\mu$-incoherent,
Problem \eqref{mc} has $\trux$ as its unique solution with high probability
when the observation probability $p$ exceeds a certain threshold.
A string of celebrated results have placed bounds on this threshold \cite{candes2010power,gross2011recovering,recht2011simpler,chen2015incoherence}.
The tightest bound available is $p>\frac{C\mu \trur \log(\inco \trur)\log(\max(\dm_1,\dm_2))}{\min(\dm_1,\dm_2)}$ for some large enough constant $C$ \cite{ding2018leave}.

If the \eqref{mclarge} is simple, then it is computationally tractable and gives a statistical robust estimator as
argued in the introduction. The strong duality, and the condition of surjective constraint map
of \eqref{mclarge} can be easily verified. Previous
work has established primal uniqueness, but not strict complementarity because the dual certificate is only approximate.
In this paper, we show that it also satisfies strict complementarity \emph{but has multiple dual solution.} Hence \eqref{mclarge}
is only primal \newcontent{simple.} Still, as discussed in the introduction, primal \newcontent{simplicity}  guarantees that
the recovery of $\xsol$ is not stymied by the optimization and measurement error.

\begin{thm}\label{thm: mc}
	Let $\dm_{\min} = \min \{\dm_1,\dm_2\}$ and $\dm_{\max} = \max \{\dm_1,\dm_2\}$. Suppose the ground truth rank $\trur$ matrix $\trux$ is $\mu$-incoherent, and each entry of it is observed with probability $p$ independently. If $p\geq C \frac{\log(\inco \trur)\trur \inco \log \dm_{\max} }{\dm_{\min}}$ for some large enough numerical constant $C>1$, then with probability at least $1-\dm_{\min}^{-c}$ for some numeric constant $c>0$, every item of the following holds
	\begin{enumerate}
		\item Problem \eqref{mclarge} is primal \newcontent{simple} and has a unique solution $\tXsol = \begin{bmatrix}
			U\Sigma U^\top &  \trux \\ \trux^\top & V\Sigma V^\top
		\end{bmatrix}$ with rank $\trur$.
		\item  It admits multiple dual solutions.
		\item It has a dual optimal solution $\tY_0$ strictly complementary to $\tXsol$ satisfying $\lambda_{n-\trur}(\tY_0)\geq \frac{3}{8}$.
	\end{enumerate}
\end{thm}

The rest of the section is devoted to the proof of this theorem.
We assume $\dm_1=\dm_2=\dm$ to simplify the presentation.
The case for rectangular matrices $\dm_1 \ne \dm_2$ follows exactly the same reasoning.
\footnote{\newcontent{Primal and dual uniqueness
		and strict complementarity can be defined for \eqref{mc} directly instead
		of for the lifted version \eqref{mclarge}.
		However, the conclusions are the same for the lifted or standard versions:
		in Section \ref{sec: Primal regularity for original mc}, we show that primal uniqueness and strict complementarity
		still hold for \eqref{mc}, and \eqref{mc} has multiple dual solution with the same probability and assumptions as Theorem \ref{thm: mc}.
		% \mnote{Need reference to appendix!}
}}

\subsection{\newcontent{Surjective constraint map} and uniqueness of primal}
The surjective constraint map property is satisfied because the constraint map $\Amap$ has
\[
A_{i,j}=\frac{e_{i}e^\top _{\dm_2+j}+(e_{i}e^\top _{\dm_2+j})^\top }{2}, (i,j)\in \Omega,
\] which are orthogonal and hence linearly independent.

It has been proved that with high probability $\trux$ is the unique solution to \eqref{mc} \cite[Theorem 2]{ding2018leave}. Hence
the matrix
\begin{equation}
	\begin{aligned}\label{eq: xsolsol}
		\tXsol = \begin{bmatrix}
			U\Sigma U^\top  & U\Sigma V^\top  \\
			V\Sigma U^\top    & V\Sigma V^\top
		\end{bmatrix}= \begin{bmatrix}
			U \\
			V
		\end{bmatrix} \Sigma \begin{bmatrix}
			U^\top  & V^\top
		\end{bmatrix}\succeq 0
	\end{aligned}
\end{equation}
is a solution to \eqref{mclarge} and has rank equal to $\rsol$. Moreover, for any solution $\tX$ of \eqref{mclarge}, because $\trux$ is the unique
solution of \eqref{mc}, it must be of the form
\[
\tX = \begin{bmatrix}
	{W}_1 & \trux \\
	\trux ^\top & {W}_2
\end{bmatrix}.
\] Since the objective value should be equal for $\tX$ and $\tXsol$: $\tr(\tX) =\tr(\tXsol) =2\nucnorm{\trux}$,
we must have $\tX = \tXsol$.
We prove this formally in Lemma \ref{lem: uniqueminnuc} using complementarity.
Hence the primal solution to \eqref{mclarge} is unique.

\subsection{Strong duality and strict complementarity}
In this section, we will construct a dual optimal solution
to assert strong duality and strict complementarity by using the following lemma
(proved in \ref{sec: verifyAssump1}):
\begin{lemma}\label{lem: Y0construction}
	Under the setting of Theorem \ref{thm: mc},  with probability at least $1-n_{\min}^{-c}$, there exists a $Y_0\in \reals^{\dm\times \dm}$ such that
	(1) $\projomega{}(Y_0) = Y_0$,
	(2) $\projt (Y_0) =UV^\top$, and
	(3) $ \opnorm{\projto(Y_0)}\leq \frac{5}{8}$.
\end{lemma}
The operator $\projt$ is the projection to the linear space $\mathcal{T}\subset\reals^{n^2}$ consisting of matrices with columns in $\range(U)$ or rows in $\range(V)$. The projection $\projt$ can be written explicitly as $\projt{Z}=UU^\top Z+ZVV^\top -UU^\top ZVV^\top$ for any $Z\in \reals^{n^2}$. The projection $\projto(Z)=Z-\projt(Z)$ is the projection on to the subspace orthogonal to $\mathcal{T}$.

Let us write down the dual of \eqref{mclarge} with variable $y_{ij},\,(i,j)\in \Omega$ for the purpose of constructing a dual solution:
\beq
\ba{ll}\label{mclargedual} %\tag{SDP Matrix-Completion}
\mbox{maximize} & \sum_{(i,j)\in \Omega} y_{ij}(\trux)_{ij} \\
\mbox{subject to} & I - \sum_{(i,j)\in \Omega} \frac{e_{i}e_{n+j}^\top +e_{n+j}e_{i}^\top }{2}y_{ij}\succeq 0.
\ea
\eeq
By introducing a variable $\tilde{Y}=\begin{bmatrix}
	0& Y\\
	Y^\top  & 0
\end{bmatrix} =\sum_{(i,j)\in \Omega} \frac{e_{i}e_{n+j}^\top +e_{n+j}e_{i}^\top }{2}y_{ij}\in \reals^{2n\times 2n}$ and $Y\in \reals^{n^2}$, the dual problem \eqref{mclargedual} is equivalent to
\beq
\ba{ll}\label{mclargedualmatrix} %\tag{SDP Matrix-Completion}
\mbox{maximize} & 2\inprod{\trux}{Y} \\
\mbox{subject to} & I - \tY \succeq 0\\
& \projomega{}(Y) = Y.
\ea
\eeq
with decision variable $\tY = \begin{bmatrix}
	0& Y\\
	Y^\top  & 0
\end{bmatrix}$.
We work with $\tY$ instead of $[y_{ij}]_{(i,j)\in \Omega}$ because working with matrices is more convenient.
We claim the dual matrix
\[
\tY_0 =\begin{bmatrix}
	0& Y_0\\
	Y_0^\top & 0
\end{bmatrix}
\]
solves the dual problem \eqref{mclargedualmatrix}. Our derivation will also show
strong duality and strict complementarity of \eqref{opt: tracemc}. We first verify that $\tY_0$ is feasible for \eqref{opt: dualtracemc}.
\paragraph{Linear feasibility:} This is due to $\projomega{}(Y_0) = Y_0$ by assumption on $Y_0$.
\paragraph{PSD feasibility and strict complementarity:} Take any $w=\begin{bmatrix}
	u \\ v
\end{bmatrix} \in  \range\left(\begin{bmatrix}
	U \\ V
\end{bmatrix}\right)$. As $\frac{1}{2} \begin{bmatrix}
	U \\ V
\end{bmatrix}[U^\top,V^\top]$ is the projection matrix to $\range\left(\begin{bmatrix}
	U \\ V
\end{bmatrix}\right)$, we have
\[
w = \left(\frac{1}{2} \begin{bmatrix}
	U \\ V
\end{bmatrix}[U^\top,V^\top ]\right) w.
\]
By expanding the righthand side of the above equality,
we reach
$
w = \begin{bmatrix}
	UV^\top v \\ VU^\top u
\end{bmatrix}.
$ Using this fact and the definition of $\mathcal{T}^\perp$, we have
\[
(I-\tY_0)w =  w - \begin{bmatrix}
	UV^\top v \\ VU^\top u
\end{bmatrix} - \begin{bmatrix}
	\projto(Y_0)v \\
	\left[\projto(Y_0)\right]^\top u
\end{bmatrix} =0.
\]
Thus the null space of $I-\tilde{Y}_0$ contains  $\range\left(\begin{bmatrix}
	U \\ V
\end{bmatrix}\right)$.
Now take any $z=\begin{bmatrix}
	z_1 \\ z_2
\end{bmatrix} \perp \range\left(\begin{bmatrix}
	U \\ V
\end{bmatrix}\right)$, then $U^\top z_1 + V^\top z_2 = 0$. Thus the quadratic form
$z^\top(I-\tY_0)z$ satisfies
\begin{equation}
	\begin{aligned}\label{eq: strictcomplmcderive}
		z^\top(I-\tY_0)z &= \twonorm{z_1}^2 +\twonorm{z_2}^2-2z_2^\top VU^\top z_1-2z_2^\top \projto(Y_0)z_1 \\
		%&\overset{(a)}{=}\twonorm{z_1}^2 +\twonorm{z_2}^2+2z_2^\topVV^\top z_2+2z_2^\top\projto(Y_0)z_1\\
		&\overset{(a)}{\geq} \twonorm{z_1}^2 +\twonorm{z_2}^2+2z_2^\top VV^\top z_2-2\twonorm{z_1}\opnorm{\projto(Y_0)}\twonorm{z_2}\\
		&\overset{(b)}{\geq} \twonorm{z_1}^2 +\twonorm{z_2}^2-\frac{5}{4}\twonorm{z_1}\twonorm{z_2}\geq \frac{3}{8}(\twonorm{z_1}^2 +\twonorm{z_2}^2),
	\end{aligned}
\end{equation}
where step $(a)$ is due to the fact $U^\top z_1 + V^\top z_2 = 0$ and step $(b)$ is because of $\opnorm{\projto(Y_0)}\leq \frac{5}{8}$. We hence have shown that $I-\tilde{Y}_0$ is PSD when restricted to the space orthogonal to $ \range\left(\begin{bmatrix}
	U \\ V
\end{bmatrix}\right)$. To conclude  $I-\tilde{Y}_0$ is PSD, we recall the null space of $I-\tilde{Y}_0$ contains  $\range \left(
\begin{bmatrix}
	U \\ V
\end{bmatrix}\right)$. Note that strict complementarity is satisfied as $\tilde{Y}$ is optimal to \eqref{mclargedual} as shown in the next paragraph:
\begin{equation}\label{eq: strictcomplmc}
	\begin{aligned}
		\rank(I-\tY_0)=n-\trur, \quad \text{and} \quad \lambda_{n-\trur}(I-\tY_0)\geq \frac{3}{8}.
	\end{aligned}
\end{equation}
\paragraph{Optimality and strong duality:} the objective value of $\tY_0$ and $\tXsol$ satisfy
\begin{equation}
	\begin{aligned}
		2\tr(\trux Y_0)\overset{(a)}{=}2\tr(\trux \projt(Y_0))\overset{(b)}{=}2\tr(V\Sigma U^\top UV^\top )=2\tr(\Sigma)\overset{(c)}{=}\tr(\tXsol).
	\end{aligned}
\end{equation}
Here step $(a)$ uses the fact that $\trux \in \mathcal{T}$. In step $(b)$, we use the fact  $\projt(Y_0)=UV$ in Lemma \ref{lem: Y0construction}. Step $(c)$ uses the form of $\tXsol$ in \eqref{eq: xsolsol}.
Thus we see $I-\tY_0$ is indeed dual optimal and satisfies strong duality.

\subsection{Multiple dual solutions}\label{sec: Multiple dual solutions}
To establish the fact that the dual has multiple solutions,
let us introduce the following lemma concerning the uniqueness of the dual.
\begin{lemma}\cite[Theorem 6, 7, and 11]{alizadeh1997complementarity}\label{lem: uniquenondegenerate}
	Suppose the primal SDP \eqref{p} is primal simple
	\footnote{We note that the results in \cite{alizadeh1997complementarity} require
		either the primal \eqref{p} or the dual \eqref{d} satisfies
		the Slater's condition. Since this condition is only used to ensure $\pval = \dval$,
		which is covered by our definition
		of strong duality, Slater's condition is no longer required.}.
	A necessary condition for the dual to be unique is $\frac{(n-\rsol)(n-\rsol+1)}{2}\leq \frac{n(n+1)}{2}-m$.
\end{lemma}

Using the necessary condition from Lemma \ref{lem: uniquenondegenerate}, we have the dual is not unique unless
\begin{equation}
	\begin{aligned} \label{eq: mcdualuniquenecessarycondition}
		\frac{(2\dm-\rsol) (2\dm-\rsol +1)}{2}\leq \frac{2\dm(2\dm+1)}{2}-m \iff  \frac{2n\rsol -(\rsol)^2 +\rsol}{2}\geq m.
	\end{aligned}
\end{equation}
However, since $m\geq \frac{1}{2}p\dm^2 \geq  C\dm \trur\inco\log (\trur\inco) \log \dm $ with probability at least $1-\dm^{-2}$ for some large
constant $C>1$ and $\rsol=\trur$ (as $\tXsol$ has rank $\trur$), we see \eqref{eq: mcdualuniquenecessarycondition} cannot be satisfied for any $n\geq 1$ and hence the dual is not unique.

\subsection{Numerical verification of multiple dual solution}
In this section, we demonstrate numerically that
the matrix completion problem \eqref{mclarge} indeed admits multiple dual solutions
by constructing two dual solutions.
The problem instance we consider is the matrix completion problem with $\dm = \dm_1=\dm_2=50$.
The original matrix $\trux$ is generated randomly with rank $\trur=2$.
We set the observation probability $p$ to be $p=3 \trur \log(\dm)/\dm$.
We compute the primal solution $\xsol$ of \eqref{mclarge} using the SDPT3 solver \cite{toh1999sdpt3} and find it
is exactly $\tXsol$.

Now we demonstrate the multiplicity of the dual solution by constructing several dual solutions.
Let $U\in \reals^{2\dm \times (2\dm - \rsol)}$ have columns that form
an orthonormal basis for $\nullspace(\tXsol)$.
Recall $\xsol$ has rank $\rsol = \rank(\tXsol)=\trur$.
For a given cost matrix $C\in \sym^{2\dm-\rsol}$,
we solve the problem
\beq
\ba{ll}\label{mcTestDualuniqueness}
\mbox{maximize} & \inprod{C}{Z_s} \\
\mbox{subject to} & Z_s \succeq 0\\
& UZ_sU^\top = I - \sum_{(i,j)\in \Omega} \frac{e_{i}e_{n+j}^\top+e_{n+j}e_{i}^\top}{2}y_{ij}\\
\ea
\eeq
with decision variable $Z_s \in \sym ^{2\dm -\rsol}$ and $y\in \reals^{m}$
We solve this problem twice, with (i) $C =I$ and (ii) $C$ having iid standard Gaussian entries.

Denote the solution of \eqref{mcTestDualuniqueness} as $Z_{s,C}$ for each different $C$.
By construction, the matrix $Z_{C}= UZ_{s,C}U^\top$ is dual optimal as it is feasible and $\tr(Z_{C}\xsol)=0$.

We plot the spectrum of $Z_C$ in Figure \ref{fig:matrixcompletionmultipledual}.
We see the spectra are quite different; clearly these two dual solutions are not the same!
\begin{figure}
	\centering
	\includegraphics[width=0.8\linewidth]{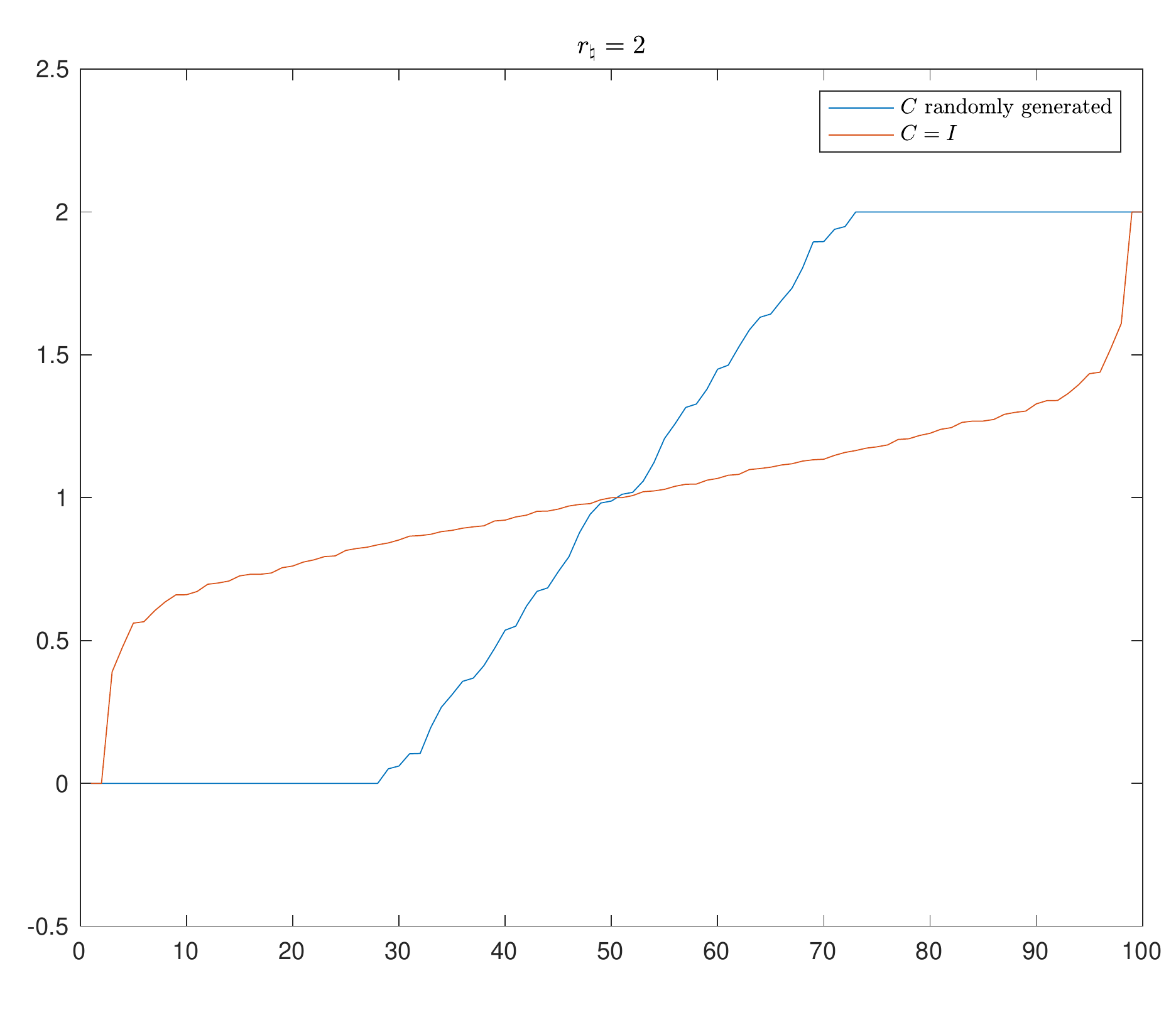}
	\caption{The eigenvalues of different dual solution $Z_{C}$.}
	\label{fig:matrixcompletionmultipledual}
\end{figure}

\newcontent{\section{Discussion and conclusion}
	In this note, we have shown that generic SDPs are simple,
	and that many structured low rank matrix recovery problems are also simple.
	Building on the framework established here,
	an important future direction is to understand whether most SDPs
	with linear inequality constraints are simple.
	These inequality constraints can be embedded into the standard form SDP
	presented here, but this embedding can often lead to a problematic increase in the dimensionality of the problem
	and give too much degree of freedom in the dual space.
	For example, one important special case concerns entrywise constraints of $X$
	such as the nonnegativity constraint $X\geq 0$.
	
	We conjecture that primal \newcontent{simplicity}  continues to hold for SDP applications in statistics and signal processing, even in the presence of nonnegativity constraints.
	Intuitively, in this applications, 
	we expect the optimal solution to coincide with the underlying signal, which means the optimal solution is likely to be unique.
	One common way to prove primal uniqueness for these problems is to
	first prove strict complementarity holds;
	see e.g. \cite[Equation (1.5) and Section 2]{li2018convex}.\footnote{Do note that the term of strict 
		complementarity is not explicitly specified.}
	%(Though note that the authors of that paper do not use the word complementarity!)
	
	On the other hand, we expect the dual solution of these problems not to be unique,
	as the number of constraints, which is the number of measurements,
	is usually a bit larger than the intrinsic dimension.
	(For example, in matrix completion,
	the intrinsic dimension of a rank $r$ matrix of size $n\times n$ is $\bigO{nr}$
	but the number of measurements needs to be greater by a log factors 
	as shown in \cref{thm: mc} which is $pn^2 = \bigO{nr\mu \log (\mu r)\log n}$).)
	The excess of constraints seems necessary to ensure successful recovery with high probability,
	but destroys dual uniqueness as there is more freedom in the dual.
	
	As a first step towards handling problems with more constraints,
	consider a simple SDP arising from community
	detection \cite[Equation (1.5)]{li2018convex},
	which is a Max-Cut SDP with an additional nonnegativity constraint
	$X\geq 0$ and admits a unique completely positive primal solution.
	As shown in
	\cite[section 2]{li2018convex}, it is primal simple but not 
	simple
	(with the appropriate generalization of simplicity and 
	primal simplicity to problems
	with inequality constraints). We leave the exact details to future work.
}

\appendix
\section{Definition of stationary points for \eqref{bm}}\label{sec: appendixBMstationary}
We define second order stationary points formally below:
\begin{definition}
	Suppose $\mathcal{M}_r = \{F\mid  \mathcal{A}(FF^\top) = b\}$ is a Riemannian manifold equipped
	with the trace inner product of $\reals^{\dm \times p}$. Denote the tangent space of
	any $F\in \mathcal{M}_r$ as $T_F\mathcal{M}_r\subset \reals ^{n\times r}$. A point
	$F\in \reals^{\dm \times r}$ is a second order stationary point if the following two conditions are
	satisfied:
	\begin{itemize}
		\item  	the Riemannian gradient $\grad f(F) \in T_F\mathcal{M}_r$ satisfies
		$\grad f(F) =0  $
		\item the Riemannian Hessian $\hess f(F)$ is a positive semidefinite symmetric linear
		map from $T_F\mathcal{M}_r$ to $T_F\mathcal{M}_r$.
	\end{itemize}
\end{definition}
Informally, we can see the conditions required of second order stationary points for \eqref{bm}
match the conditions required in the unconstrained case.
Ideas from Riemmanian optimization makes the gradient and Hessian in the constrained setting precise and rigorous. %We note the reader that the condition for $\mathcal{M}_r$ being a Riemannian manifold, the exact formula of the Riemannian gradient and
%Hessian are not important to our disccusion here.
We refer the reader to \cite{absil2009optimization} for the general definition of Riemannian gradient and Hessian for smooth functions defined
on Riemmannian manifold. For the condition guaranteeing $\mathcal{M}_r$ being a Riemannian manifold, the detailed formula of the tangent space
$T_F\mathcal{M}_r$, the $\grad f(F)$, and $\hess f(F)$, see \cite[Assumption 1.1, Equations (3), (7), and (10)]{boumal2018deterministic} respectively.
\section{Lemma for Section \ref{sec: mc}}\label{sec: apformc}
\begin{lemma} \label{lem: uniqueminnuc} Fix a rank $\trur$ matrix $\trux \in \reals^{\dm_1\times \dm_2}$ with singular value
	decomposition $X= U\Sigma V^\top$. Here $U\in \reals^{\dm_1\times \trur}$, $V\in \reals^{\dm_2\times \trur}$ have orthonormal columns and $\Sigma\in \sym^{\trur}$ is diagonal.
	The optimization problem with decision variable $W_1\in \sym ^{\dm_1}$ and $W_2\in \sym ^{\dm_2}$
	\beq \label{opt: tracemc}
	\ba{ll}
	\mbox{minimize} & \tr(W_1)+\tr(W_2)\\
	\mbox{subject to} & \begin{bmatrix}
		W_1 & \trux \\
		\trux ^\top & W_2
	\end{bmatrix}\succeq 0
	\ea
	\eeq
	has a unique solution $W_1 =U\Sigma U^\top$ and $W_2 = V\Sigma V^\top$ with optimal value $2\nucnorm{\trux}$.
\end{lemma}
\begin{proof}
	The dual problem of Problem \eqref{opt: tracemc} is simply
	\beq \label{opt: dualtracemc}
	\ba{ll}
	\mbox{maximize} & -2\tr(\trux ^\top Z)\\
	\mbox{subject to} & \begin{bmatrix}
		I  & Z\\
		Z^\top & I
	\end{bmatrix}\succeq 0,
	\ea
	\eeq
	with decision variable $Z\in \reals^{\dm_1\times \dm_2}$. First take $\zsoloy  = -UV^\top$ and $W^\star_1 =U\Sigma U^\top$ and $W^\star_2 = V\Sigma V^\top$. It can be easily verified that $\zsoloy$, $W^\star_1$ and $W^\star_2$ are feasible. We also find that
	\[
	\tr(W^\star_1)+\tr(W^\star_2)-(-2\tr(\trux^\top \zsoloy))=2\tr(\Sigma)-2\tr(\Sigma)=0.
	\]
	Hence both $\zsoloy$ and $W^\star_1=U\Sigma U^\top, W^\star_2 = V\Sigma V^\top$ are optimal, and the optimal value of
	\eqref{opt: tracemc} is $2\tr(\Sigma)=2\nucnorm{\trux}$. Now take any $\bar{W}_1$ and $\bar{W}_2$ that is optimal to
	\eqref{opt: tracemc}. Using the optimality of
	$\begin{bmatrix}
		I  & \zsoloy\\
		\zsoloy & I
	\end{bmatrix}\succeq 0$,  we have
	\[
	0 = \begin{bmatrix}
		I  & \zsoloy\\
		\zsoloy ^\top& I
	\end{bmatrix}
	\begin{bmatrix}
		\bar{W}_1 &\trux \\
		\trux^\top& \bar{W}_2
	\end{bmatrix}
	=\begin{bmatrix}
		\bar{W}_1 -U\Sigma U^\top  & \trux -UV^\top\bar{W}_2\\
		-VU^\top\bar{W}_1+\trux ^\top & -V\Sigma V^\top+\bar{W}_2
	\end{bmatrix}.
	\]Thus we must have $\bar{W}_1 = U\Sigma U^\top$ and $\bar{W}_2 =V\Sigma V^\top$.
\end{proof}
\subsection{Proof of Lemma \ref{lem: Y0construction}} \label{sec: verifyAssump1}
The matrix $Y_0$ is actually a dual certificate for $\trux$ for \eqref{mc}. We follow the construction procedure in \cite{ding2018leave}.

First set $k_0 \defn C_0\log (\inco \trur)$ for some large enough numerical constant $C_0$. %Instead of observing each entry of $\trux$ independently,
We can suppose (without loss of generality) that the set $\ob$ of observed entries is generated from $\ob = \cup_{t=1}^{k_0} \ob_t$, where for each $t$ and matrix index $(i,j)$, $\Prob[(i,j)\in \ob_t] = q \defn 1 -(1-p)^{\frac{1}{k_0}}$, and the event $\{(i,j)\in \ob_t\}$ is independent of all others.
Denote the projection $ \projomega{t} $ by $[\projomega{t}(Z)]_{ij}= Z_{ij} \indic\{(i,j) \in \ob_t \}$, where
$\indic\{(i,j) \in \ob_t \} =1$ if $(i,j)\in \ob_t$ and $0$ otherwise. We also denote $\rprojk{t} \defn \frac{1}{q}\projomega{t}$.

We use independent samples in constructing $k_0$ building blocks of the first piece of the dual certificate: set  $W^0 \defn  UV^T$ and
\begin{align}\label{eq:defW}
	W^t \defn \projt \hpk{t} (W^{t-1}), \quad t =1,2,\dots, k_0 -1,
\end{align}
where $\hpk{t} = \Id - \frac{1}{q}\projomega{t}$, and $\Id$ is the identity map on $\reals^{\dm \times \dm}$.

We then use the same sample set $ \ob_{k_0}  $ in the next $\To \defn  2\log \dm +2 $ building blocks of the second piece of the dual
certificate: set  $Z^{0} = W^{k_0-1}$ and
\begin{align}\label{eq:defZ}
	Z^{t} \defn \projt \hpk{k_0} (Z^{t-1}) = ( \projt \hpk{k_0} )^t(W^{k_0-1}), \quad t = 1,2,\dots, \To -1.
\end{align}

The building block of the last piece is simply running \eqref{eq:defZ} for all $t\geq \To$.
The final dual certificate $ Y $ is constructed by summing up the above iterates: set
\begin{align}\label{eq: defofY}
	Y_1 \defn  \sum_{t=1}^{k_0-1} \rprojk{t}\projt(W^{t-1}),
	\,
	Y_2 \defn \sum_{t=1}^{\To} \rprojk{k_0}\projt(Z^{t-1}),
	\,
	Y_3 \defn \sum_{t=\To+1}^\infty \rprojk{k_0}\projt(Z^{t-1})
\end{align}
and our desired $Y$ is simply
\[
Y\defn Y_1+Y_2+Y_3.
\]
\paragraph{Convergence of $Y_3$} We first verify the infinite series $Y_3$ indeed converges. Denote the Frobenius norm as $\fronorm{\cdot}$. Using \cite[Theorem 4.1]{candes2009exact}, we have $\opnorm{\projt \hpk{k_0}\projt}\leq \frac{1}{4}$
%
%=\max_{\fronorm{A}=1}\fronorm{\projt \hpk{k_0}\projt(A)}\leq \frac{1}{4}$ and hence
%
for all $t\geq 1$
\begin{align}\label{eq: convergenceofZtnorm}
	\fronorm{Z^{t}}\leq \opnorm{\projt \hpk{k_0}\projt}\fronorm{Z^{t-1}}\leq \frac{1}{4}\fronorm{Z^{t-1}}.
\end{align}
Hence the series $\fronorm{Y_3}\leq\frac{1}{p} \fronorm{Z^{\To}}\sum_{t=\To+1}^\infty \frac{1}{4^t} $, and the infinite series in $Y_3$ indeed converges.
\paragraph{The condition $\projomega{}(Y)=Y$} Note that  $\projomega{}\rprojk{t}=\rprojk{t}$ for any $t$ by construction. Hence using the convergence of series in $Y_3$ and $\projomega{}$ is a continuous map, we reach  $\projomega{}(Y)=Y$
\paragraph{The condition $\projt(Y)=UV^\top $} Using the construction of $Y$, we find that
$\projt(Y_1+Y_2+ \sum_{\tau = \To+1}^t \rprojk{k_0}\projt(Z^{\tau -1}))-UV^\top= -Z^t$.
Hence, we have
\begin{equation}
	\begin{aligned}
		\fronorm{\projt(Y)-UV^\top}&=\lim_{t\rightarrow \infty}\fronorm{\projt(Y_1+Y_2+ \sum_{\tau =\To+1}^t \rprojk{k_0}\projt(Z^{\tau-1}))-UV^\top}\\
		&=\lim_{t\rightarrow \infty}\fronorm{Z^t}.
	\end{aligned}
\end{equation}
Now using \eqref{eq: convergenceofZtnorm}, we see the above is actually $0$ and hence  $\projt(Y)=UV^\top$.
\paragraph{The condition $\opnorm{\projto(Y)}\leq \frac{5}{8}$} In \cite[Section 6, ``Validating
Condition 2(a)”, pp 30-31]{ding2018leave}, it has been shown that
$\opnorm{\projt{(Y_1+Y_2)}}\leq \frac{1}{2}$. Using \cite[Inequality (92)]{ding2018leave}, we have $\fronorm{Z^{t_0}}\leq \frac{1}{4\dm}$ and hence $\fronorm{Y_3}\leq \frac{1}{2p\dm}<\frac{1}{8}$.
Thus $\opnorm{\projto(Y_3)}\leq \fronorm{\projto(Y_3)}\leq \fronorm{Y_3}\leq \frac{1}{8}$. Hence the operator norm of $\projto(Y)$ satisfies
$\opnorm{\projto(Y)}\leq\opnorm{\projto(Y_1+Y_2)}+\opnorm{\projto(Y_3)}\leq \frac{5}{8}$.
\newcontent{\subsection{Primal \newcontent{simplicity}  for \eqref{mc}}\label{sec: Primal regularity for original mc}
	
	Here we show that primal uniqueness and strict complementarity hold for \eqref{mc},
	yet the problem has multiple dual solutions,
	exactly like the lifted version \eqref{mclarge}.
	
	To show the problem has multiple dual solutions,
	note the Lagrangian dual of \eqref{mc} is
	\beq
	\ba{ll}\label{mcdualmatrix}
	\mbox{maximize} & \inprod{\trux}{\projomega{}(Y)} \\
	\mbox{subject to} & \opnorm{\projomega{}(Y)}\leq 1.
	\ea
	\eeq
	Equation \eqref{mcdualmatrix} is equivalent to equation \eqref{mclargedualmatrix} in the following sense:
	if $\tY = \begin{bmatrix}
		0& Y\\
		Y^\top  & 0
	\end{bmatrix}$ is optimal for \eqref{mclargedualmatrix}, then $Y$ is optimal for \eqref{mcdualmatrix},
	and vice versa. This is a simple consequence of $\opnorm{\tY}= \opnorm{Y}\leq 1\iff I\succeq\tY$, and the
	constraint $\projomega{}(Y)=Y$.
	Hence multiple dual solutions to \eqref{mclargedualmatrix}
	implies multiple solutions of \eqref{mcdualmatrix},
	and hence  \eqref{mc} has multiple dual solutions by Theorem \ref{thm: mc}.
	
	A proof of primal uniqueness appears in \cite[Theorem 2]{ding2018leave}.
	
	Finally, we show that strict complementarity holds. Strict complementarity 
	in this context \footnote{\newcontent{The definition of 
			strict complementarity is inspired from the dual strict complementarity defined in \cite[Section 4]{drusvyatskiy2018error} and 
			the equality in \cite[Equation (49) and (50)]{zhou2017unified}.}}
	means there exists $Y=\projomega{}(Y)$ such that $Y\in \ri\left(\partial \nucnorm{\trux}\right)$,
	where $\ri(\cdot)$ extracts the relative interior of its argument. The existence of 
	such $Y$ is ensured by \cref{lem: Y0construction}.
}
\section*{Acknowledgments}
We would like to thank Yudong Chen,
James Renegar, Adrian Lewis, and Michael L. Overton for helpful discussions.
We would also like to thank the editor and anonymous reviewers for their feedback.

\bibliographystyle{alpha}
\bibliography{references}
\end{document}